\documentclass{article}
\usepackage[a4paper]{geometry}
\usepackage{amsmath,amssymb,amsthm}
\usepackage{pict2e}
\usepackage{dsfont}
\usepackage{graphicx}
\usepackage{overpic}
\usepackage{paralist}
\usepackage{mathtools}

\mathtoolsset{centercolon}

\newcommand{\R}{\mathbb R}
\newcommand{\N}{\mathbb N}
\newcommand{\C}{\mathbb C}
\newcommand{\Z}{\mathbb{Z}}
\newcommand{\A}{\mathcal{A}}

\newcommand{\F}{\mathcal{F}}

\newcommand{\M}{\mathcal{M}}
\newcommand{\OO}{\mathcal{O}}
\renewcommand{\P}{\mathcal{P}}
\newcommand{\Q}{\mathcal{Q}}
\newcommand{\T}{\mathcal{T}}
\newcommand{\U}{\mathcal{U}}
\newcommand{\V}{\mathcal{V}}
\newcommand{\al}{\alpha}
\newcommand{\be}{\beta}
\newcommand{\de}{\delta}

\newcommand{\g}{\gamma}
\newcommand{\la}{\lambda}

\newcommand{\s}{\sigma}
\newcommand{\Om}{\Omega}
\newcommand{\om}{\omega}
\newcommand{\e}{\eta}
\newcommand{\ep}{\varepsilon}
\newcommand{\del}{\partial}
\newcommand{\1}{\mathds{1}}
\newcommand{\2}{\tfrac{1}{2}}

\renewcommand{\k}{\mathbf{k}}
\renewcommand{\l}{\mathbf{l}}

\newcommand{\nxn}{{n\times n}}
\newcommand{\dxd}{{d\times d}}
\newcommand{\ti}[1]{\tilde{#1}}
\newcommand{\wti}[1]{\widetilde{#1}}

\newcommand{\Ceins}{C_1}
\newcommand{\Czwei}{C_2}
\newcommand{\Cdrei}{C_3}
\newcommand{\Cvier}{C_4}

\renewcommand{\t}{\square} 		
\newcommand{\tit}{\ti{\t}}		
\newcommand{\st}{\boxtimes} 		
\newcommand{\Un}[1]{{\cup\;\!\!#1}}	
\newcommand{\UT}{\Un{\T}}		
\newcommand{\typ}[1]{\imath(#1)}	
\newcommand{\typt}{{\typ{\t}}}		
\newcommand{\typts}{{\scriptscriptstyle{\imath\!\!\:(\!\!\:\t\!\!\:)}}}		
\newcommand{\typv}[1]{\jmath(#1)}	
\newcommand{\I}{\Lambda_N}		
\newcommand{\Ieins}{\Lambda_1}		
\newcommand{\Pe}{\P^{\text{ext}}}	
\newcommand{\hloc}{H_{\text{loc}}}	
\newcommand{\exi}{\,\exists\,}
\newcommand{\fa}{\,\forall\,}
\newcommand{\bsn}{{\be,\s,m,N}}
\newcommand{\sn}{{\s,m,N}}
\newcommand{\sta}{\varphi}		
\newcommand{\lle}{\,\le\,}
\newcommand{\gge}{\,\ge\,}
\newcommand{\aasymp}{\,\asymp\,}
\newcommand{\sumigt}{\textstyle\sum\limits_{i\in I}\displaystyle\!\g_i|\T^i|}
\newcommand{\trees}[1]{\mathbf{T}^\Sigma_{#1}}	
\newcommand{\no}{\e}			
\DeclareMathOperator{\Sim}{Sim}
\newcommand{\Uept}[1][i]{\mathcal{N}_\ep(\st^{#1})}
\newcommand{\Urt}[1][i]{\mathcal{N}_{2r}(\st^#1)}
\newcommand{\SEept}[1][i]{\mathrm{SE}_\ep(\st^{#1})}

\newcommand{\co}{{c_0}}
\newcommand{\ci}{{c_1}}
\newcommand{\cii}{{c_2^{\scriptscriptstyle\R}}}
\newcommand{\ciii}{{c_3}}
\newcommand{\cv}{{c_8}} \newcommand{\cva}{c_8} 
\newcommand{\cvi}{{c_9}}
\newcommand{\cvii}{{c_{31}}}
\newcommand{\cix}{{c_{10}}}
\newcommand{\cx}{{c_{11}}}
\newcommand{\cxi}{{c_{15}}}
\newcommand{\cxc}{{c_{13}}}
\newcommand{\cxiii}{{c_{17}}}
\newcommand{\cxiv}{{c_{14}}}
\newcommand{\cxv}{{c_{16}}}
\newcommand{\cxvi}{{c_{12}}}
\newcommand{\cxvii}{{c_{18}}}
\newcommand{\cxviii}{{c_{19}^{\scriptscriptstyle\R}}}
\newcommand{\cxix}{{c_{20}}}
\newcommand{\cxx}{{c_{21}}}
\newcommand{\cxxii}{{c_{22}}}
\newcommand{\cxxiii}{{c_{23}}}
\newcommand{\cxxiv}{{c_{24}}}
\newcommand{\cxxv}{{c_{25}^{\scriptscriptstyle\R}}}
\newcommand{\cxxvi}{{c_7}}
\newcommand{\cxxvii}{{c_4}}
\newcommand{\cxxviii}{{c_5}}
\newcommand{\cxxix}{{c_6^{\scriptscriptstyle\R}}}
\newcommand{\cxxx}{c_{30}}
\newcommand{\cxxxii}{{c_{32}}}
\newcommand{\cxxxiii}{{c_{33}}}
\newcommand{\cxxxiv}{{c_{34}}}
\newcommand{\cxxxv}{{c_{35}}}
\newcommand{\cxxxvi}{{c_{36}^{\scriptscriptstyle\R}}}
\newcommand{\cxxxvii}{{c_{37}}}
\newcommand{\cxxxviii}{{c_{38}}}
\newcommand{\cxlii}{{c_{42}}}
\newcommand{\cxliii}{{c_{41}}}
\newcommand{\cxliv}{{c_{39}}}
\newcommand{\cxlv}{{c_{43}^i}}
\newcommand{\cxlvi}{{c_{40}}}
\newcommand{\cxlvii}{{c_{44}}}
\newcommand{\cxlviii}{{c_{45}}} 
\newcommand{\cxlix}{{c_{46}}}
\newcommand{\cl}{{c_{47}^{\scriptscriptstyle\R}}}
\newcommand{\cli}{{c_{26}}}
\newcommand{\clii}{{c_{27}}}
\newcommand{\cliii}{{c_{28}}}
\newcommand{\cliv}{{c_{29}}}
\newcommand{\clv}{{c_{48}}}		
\newcommand{\clvi}{{c_{49}}}	

\theoremstyle{plain}
 \newtheorem{thm}{Theorem}[section]
 \newtheorem{lem}[thm]{Lemma}
 \newtheorem{corol}[thm]{Corollary}
\theoremstyle{definition}
 \newtheorem{rem}[thm]{Remark}
 
\theoremstyle{remark}
 \newtheorem*{rem*}{Remark}

\DeclareMathOperator{\dist}{dist}
\DeclareMathOperator{\hull}{hull}
\DeclareMathOperator{\SO}{SO}
\DeclareMathOperator{\interior}{int}
\DeclareMathOperator{\image}{im}

\newcommand{\sod}{\SO(d)}
\newcommand{\sot}{\SO(2)}

\title{Spontaneous Breaking of Rotational Symmetry with Arbitrary Defects and a Rigidity Estimate}
\author{Simon Aumann\\[0.5ex] \textit{\small{Mathematisches Institut, Ludwig-Maximilians-Universit\"at M\"unchen}}\\[-0.5ex] \textit{\small{Theresienstr.\ 39, D-80333 M\"unchen, Germany}}\\[-0.5ex] \small{aumann@math.lmu.de}}

\begin{document}

\maketitle

\begin{abstract}
 The goal of this paper is twofold. First we prove a rigidity estimate, which generalises the theorem on geometric rigidity of Friesecke, James and M\"uller to 1-forms with non-vanishing exterior derivative.
 
 Second we use this estimate to prove a kind of spontaneous breaking of rotational symmetry for some models of crystals, which allow almost all kinds of defects, including unbounded defects as well as edge, screw and mixed dislocations, i.e.\ defects with Burgers vectors.
\end{abstract}

\begin{small}
\end{small}

\section{Introduction}

Condensed matters in solid state usually have the structure of a crystal: The molecules are arranged in some regular pattern. Real crystals are in fact not perfectly regular, but form a perturbation of the pattern. They also have defects. One can describe a crystal using the fundamental approach of statistical mechanics. Some probability distributions determine the location of the molecules. Their local interaction should specify the distribution. One wants to extract the global behaviour of the crystal from these local interactions. This is not well understood in a mathematically rigorous sense yet.

One question to tackle is whether the crystal globally preserves or breaks symmetries of the local interactions. Richthammer showed that the translational symmetry is preserved in a quite general two-dimensional setting, see \cite{r7}. But in the case of rotational symmetry one expects a different outcome: rotational symmetry should be broken. Merkl and Rolles showed this for a toy model of a crystal without defects in \cite{mr9}. This was extended by Heydenreich, Merkl and Rolles in \cite{hmr13} to a model which allows simple defects.

In the present work, it is shown that the rotational symmetry is broken (in a weaker form) for a class of models where almost all kinds of defects are allowed. Let us describe this class informally. A model consists of a tessellation, some local Hamiltonians, a measure for the surface of the defects and some parameters. The crystal shall have a favourite structure, which depends on the considered matter and is described by the tessellation. Thus the molecules form locally a perturbation of the tessellation. A local perturbation costs some energy, which is described by the local Hamiltonians. As already mentioned, the crystal may have various defects. In particular, there may be edge, screw and mixed dislocations, i.e.\ defects with Burgers vectors, as well as large unbounded defects. We only require that the size of a defect is larger than an arbitrary small, but fixed number. A defect is punished proportional to the size of its surface. This can be interpreted as a surface tension. Moreover, there is a chemical potential which favours a large number of molecules. 

Let us be a bit more precise. The crystal lives in a $d$-dimensional box ($d\ge2$) of size $N$ (with periodic boundary), and the centre of the molecules are given by a random set $\P$ of points in the box. A point configuration $\P$ determines a set $\T$ of tiles, which are locally a perturbation of the tessellation. Furthermore, it determines the quantity $S$ measuring the surface of the defects. The local Hamiltonian $\hloc(\t)$ gives the energy costs of the perturbed tile $\t$ in any way which fulfil a reasonable inequality. Then the global Hamiltonian is defined by
$$ H_\sn(\P) \,:=\, \sum_{\t\in\T} \hloc(\t) + \s S - m|\P| $$
with $\s>0$ and $m\in\R$. The three addends describe the local perturbation, the surface energy and a chemical potential. Using a Possion Point Process $\mu$ in the box as reference measure, the probability measure $P_\bsn$ is given by
$$ dP_\bsn \,:=\, \frac{1}{Z_\bsn} e^{-\be H_\bsn}\,d\mu $$
with inverse temperature $\be>0$ and partition sum $Z_\bsn$. Then we show (Theorem~\ref{thm:symmetrybreaking}) that there exists $\s_0(N,m)\asymp N^2+m$ such that for all $m\ge m_0$
$$\adjustlimits\lim_{\be\to\infty} \limsup_{N\to\infty} \sup_{\s\ge\s_0(N,m)} 
  E_\bsn\bigg[ \inf_{R\in\sod} \frac{1}{|\T|}\sum_{\t\in\T} \|V-R\|^2_{L^2(\t)}\bigg] \;=\; 0\,, $$
where $V:\UT\to\R^\dxd$ measures point-wise the deformation (rotation and scaling) of the crystal. Thus the crystal is globally close to a constant rotation $R\in\sod$, i.e.\ there is a long-range order in the crystal. But if the local Hamiltonians and the surface measure are chosen rotational invariant, which is possible and reasonable, the global Hamiltonian is rotational symmetric. Therefore the rotational symmetry is broken.

In order to prove this result, we follow the approach of Heydenreich, Merkl and Rolles. Their main ingredient is the theorem on geometric rigidity of Friesecke, James and M\"uller \cite[Theorem~3.1]{fjm2}. We first prove a more general rigidity estimate described below and apply it to prove the result stated above, using a more or less similar technique as Heydenreich, Merkl and Rolles. 

The main constraint of our theorem is that the limit is not uniform in the box size: $\s_0$ depends on $N$. But with the chosen method this is the best possible result, since one constant in the rigidity estimate is not scale-invariant. In order to get results uniform in the size of the box, one might have to use much more involved approaches like renormalisation. 

The already mentioned rigidity estimate is the other goal of this article. Results on geometric rigidity go back to a theorem of Liouville. It states that if the derivative of a smooth function $v:\R^d\supseteq M\to\R^d$ is point-wise a rotation, then the function is globally a rigid motion, i.e.\ its derivative is everywhere the same rotation.  A major step further was the now classical rigidity estimate of Friesecke, James and M\"uller \cite[Theorem~3.1]{fjm2}. They bounded the $L^2$-distance of the derivative from a constant rotation by a constant times the $L^2$-distance from the whole rotation group $\sod$. This was further generalised by M\"uller, Scardia and Zeppieri to fields with non-zero curl, at least in dimension $d=2$, see \cite[Theorem~3.3]{msz13}.

Here we consider matrix-valued functions $V:M\to\R^\dxd$ on an open, connected and bounded set $M\subset\R^d$ with smooth boundary in dimension $d\ge2$. We also identify such a function line by line with a vector of $1$-forms. We show (Theorem~\ref{thm:rigidity}) that the $L^2$-distance of $V$ from a single constant rotation $R\in\sod$ is bounded by the sum of a constant times the $L^2$-distance of $V$ from the rotation group $\sod$ and another constant times the $L^p$-norm (with $p\ge 2d/(2+d)$) of the (component-wise) exterior derivative $dV$ of $V$. We also determine the scaling of the constants (Lemma~\ref{lem:scale}). Note that one of them is not scale-invariant. If $V=dv$ for some function $v:M\to\R^d$ (which implies $dV=0$), this estimate reduces to \cite[Theorem~3.1]{fjm2}. It is also an extension of \cite[Theorem~3.3]{msz13}, which handles the case $d=2$ and $p=1$. 

This rigidity estimate is the content of Chapter~\ref{cha:rigidity}. In Chapter~\ref{cha:symmetrybreaking} we state the considered class of crystal models accurately and prove the result on the spontaneous breaking of the rotational symmetry. Finally we give two examples of concrete models. First we consider the two-dimensional triangular lattice. This yields a model analogous to the model considered in \cite{hmr13}. Then we draw our attention to a crystal whose favourite structure is the $d$-dimensional cubic lattice.

\section{A Rigidity Estimate} \label{cha:rigidity}

\subsection{Statement of the Rigidity Estimate}

Let $d\ge2$. We work with functions mapping to $\R^\dxd$ defined on an open, connected and bounded set $M\subset\R^d$ with smooth boundary. We identify such a matrix-valued function $V=(V_{ij})_{1\le i,j \le d}$ line by line with a vector $V=(V_i)_{1\le i\le d}$ of 1-forms $V_i=\sum_{j=1}^d V_{ij}dx_j$. Then the exterior derivative $dV=(dV_i)_{1\le i\le d}$ is a vector of 2-forms with components $dV_i=\sum_{k<l} (\del_kV_{il}-\del_lV_{ik})dx_k\wedge dx_l$ if the derivatives exist. For $p\ge1$, its $p$-norm is defined by
$$ \|dV_i\|_{L^p(M)}^p := \sum_{k<l} \big\|\del_kV_{il}-\del_lV_{ik}\big\|_{L^p(M)}^p 
\qquad\text{and}\qquad
 \|dV\|_{L^p(M)}^p := \sum_{i=1}^d \|dV_i\|_{L^p(M)}^p \,.$$
We say that $V\in L^2(M,\R^\dxd)$ satisfies $dV\in L^p(M)$ for some $p\ge1$ if there exist smooth functions $V^n\in L^2(M,\R^\dxd)$, $n\in\N$, such that $V^n\to V$ in $L^2$ as $n\to\infty$ and such that $(dV^n)_{n\in\N}$ is a Cauchy sequence in $L^p$. In that case we define $dV:=L^p\text{-}\lim_{n\to\infty} dV^n$. This limit is well-defined by the following remark.

\begin{rem*}
 For $k\in\N_0$, let $L^2\Om^k(M)$ denote the space of $k$-forms on $M$ whose coefficients are in $L^2(M)$. Other spaces of $k$-forms are defined analogously. Let $\nu \in L^2\Om^1(M)$ be a 1-form. Then a $2$-form $\om$ is the exterior derivative of $\nu$ in the weak sense, i.e.\ $d\nu=\om$, if $\langle\nu,\de\chi\rangle=\langle\om,\chi\rangle$ holds for all $2$-forms $\chi\in C_c^\infty\Om^2(M)$, where the codifferential $\de$ is the adjoint operator to $d$. Therefore the weak exterior derivative is unique. 
 In particular, if there are smooth $1$-forms $\nu_n\in C^\infty\Om^1(M)$ such that $\nu_n\to\nu$ in $L^2$ and $d\nu_n\to\psi$ in $L^p$ for a $2$-form $\psi\in L^p\Om^2(M)$, then $\psi=\om=d\nu$. Thus the limit is well-defined.

 Note that we did not require that the weak exterior derivative $\om$ of $\nu$ is in $L^p$, but we imposed the possibly stronger condition that we can approximate $\nu$ with smooth $1$-forms whose exterior derivatives converge in $L^p$. It is not relevant for our purposes whether these two conditions are equivalent. 
\end{rem*}

Now we can state the rigidity estimate of this paper.

\begin{thm} \label{thm:rigidity}
 Let $d\ge2$ and $M\subset\R^d$ be open, connected and bounded with smooth boundary. Let further $p\ge2d/(2+d)$.
 Then there exist constants $\Ceins=\Ceins(M)$ and $\Czwei=\Czwei(M,p)$ such that for all $V\in L^2(M,\R^\dxd)$ with $dV\in L^p(M)$ 
 there exists a rotation $R\in \sod$ with
 $$ \|V-R\|_{L^2(M)} \,\le\,
    \Ceins \|\dist(V,\sod)\|_{L^2(M)} +
    \Czwei \|dV\|_{L^p(M)} \,.$$
\end{thm}
\noindent Theorem~\ref{thm:rigidity} also holds if $M$ is a finite box with periodic boundary conditions:
\begin{corol} \label{cor:rigidity}
 Let $[M]$ be a $d$-dimensional torus with $d\ge2$. Let further $p\ge2d/(2+d)$.
 Then there exist constants $\Ceins=\Ceins([M])$ and $\Czwei=\Czwei([M],p)$ such that for all $V\in L^2([M],\R^\dxd)$ with $dV\in L^p([M])$ 
 there exists a rotation $R\in \sod$ with
 $$ \|V-R\|_{L^2([M])} \,\le\,
    \Ceins \|\dist(V,\sod)\|_{L^2([M])} +
    \Czwei \|dV\|_{L^p([M])} \,.$$
\end{corol}
\begin{rem}
 The formulation of Theorem~\ref{thm:rigidity} is not the most general one. It should also hold if $M$ is an open, connected and bounded set with a more general boundary. In the proof, we  will apply Lemma~3.2.1 in the book \cite{s95} of Schwarz. He considers manifolds with smooth boundary. Though not formally stated, his results also hold if the boundary is only piecewise smooth. In \cite{mmm8} Mitrea, Mitrea and Monniaux considered similar problems as in \cite{s95}, but for domains with Lipschitz boundary. Unfortunately they do not state the exact lemma we need. Since a smooth boundary is sufficient for our purposes, we stick to that case, where the needed lemma is explicitly stated in the literature.

 It is also possible to generalise Theorem~\ref{thm:rigidity} in another direction. If $M$ is a flat manifold, which means that all transition maps are just translations, then it makes sense to speak about global rotations. Theorem~\ref{thm:rigidity} immediately generalises to compact connected flat manifolds using a straightforward generalisation of Lemma~\ref{lem:dVgleich0} to such manifolds.       
\end{rem}

We also determine the scaling of the constants in the theorem and in the corollary above.
\begin{lem} \label{lem:scale}
 Assume that Theorem~\ref{thm:rigidity} holds on $M\subseteq\R^d$, $d\ge2$, for some $p\ge1$ with constants $\Ceins(M)$ and $\Czwei(M,p)$. Let $\e>0$. 
 Then Theorem~\ref{thm:rigidity} holds on $\e M$ for $p$ with constants
 $$ \Ceins(\e M) \,=\, \Ceins(M) \qquad\text{and}\qquad \Czwei(\e M,p) \,=\,  \e^{\frac{d}{2}-\frac{d}{p}+1} \,\Czwei(M,p) \,.$$
 The same statement is true if $M\equiv[M]$ is a torus as in Corollary~\ref{cor:rigidity}.
\end{lem}
\noindent Therefore $\Ceins$ is scale invariant, but $\Czwei$ is not (except if $p=2d/(2+d)$). These scaling properties will become relevant in Section~\ref{cha:symmetrybreaking}.
\begin{rem}
 The assumption $p\ge2d/(2+d)$ is best possible. Indeed, if we had $p<2d/(2+d)$, then $dp-2d+2p<0$, which is equivalent to $\frac{d}{2}-\frac{d}{p}+1<0$. Thus, by Lemma~\ref{lem:scale}, the constant $\Czwei(\e M,p)$ would tend to zero as $\e\to\infty$. But the latter is impossible.
 
 Indeed, consider some smooth $V:\R^d\to\R^\dxd$ such that first $V(x)\in \sod$ for all $x\in\R^d$, second $V(x)=R_0$ for all $x\in\R^d$ with $|x|\ge1$ (for some fixed $R_0\in \sod$) and third $V$ being not constant on $B_1(0)$. Then  $\|dV\|_{L^p(B_1(0))}>0$ by Liouville's Theorem. Let $M=B_1(0)$ and let $\e$ be large. Then $\inf_{R\in\sod}\|V-R\|_{L^2(\e M)}\ge c$ for some constant $c>0$ since its argmin converges to $R_0$. Moreover, $dV(x)=0$ for $|x|>1$, which implies that $\|dV\|_{L^p(\e M)}=\|dV\|_{L^p(B_1(0))}\in(0,\infty)$ is constant (for $\e>1$). Theorem~\ref{thm:rigidity} states that $0<c/\|dV\|_{L^p(B_1(0))} \le \Czwei(\e M,p)$. Therefore $\Czwei(\e M,p)\to0$ as $\e\to\infty$ is indeed impossible.
\end{rem}

\subsection{Proof of the Rigidity Estimate}

Let $A\subseteq \R^d$ such that $B\subseteq A \subseteq\overline{B}$ for an open set $B\subseteq\R^d$ (where $\overline{B}$ denotes the closure of $B$). Let further $n\in\N$, $k\in\N_0$ and $p\ge1$. Then $W^{k,p}(A,\R^n)$ denotes the Sobolev space of functions $f:A\to\R^n$ such that all partial derivatives up to order $k$ exist in the weak sense and have finite $p$-norm. In particular, $W^{0,2}(A,\R^n)=L^2(A,\R^n)$. 
 
For the proof of the rigidity estimate, we use a covering argument. Therefore we need
\begin{lem} \label{lem:U1cupU2}
 Let $A_1,A_2\subseteq\R^d$ such that $B_j\subseteq A_j \subseteq\overline{B_j}$ for an open set $B_j\subseteq\R^d$, $j\in\{1,2\}$, and $\la(A_1\cap A_2)>0$ and $\la(A_2)<\infty$, where $\la$ denotes the Lebesgue-measure. Assume that, for $j\in\{1,2\}$, there exists a constant  $c_j>0$ such that for all $V\in W^{1,2}(A_j,\R^\dxd)$ with $dV=0$ there exists a rotation $R_j\in \sod$ with 
 $$ \|V-R_j\|_{L^2(A_j)} \le c_j \|\dist(V,\sod)\|_{L^2(A_j)} \,.$$
 
 Then there exists a constant $C>0$ such that for all $V\in W^{1,2}(A_1\cup A_2,\R^\dxd)$ with $dV=0$ there exists a rotation $R\in \sod$ with 
 $$ \|V-R\|_{L^2(A_1\cup A_2)} \le C \|\dist(V,\sod)\|_{L^2(A_1\cup A_2)} \,.$$ 
\end{lem}
\begin{proof}
 We set
 $$ C\,=\,\sqrt{\Big(\frac{4\la(A_2)}{\la(A_1\cap A_2)}+2\Big)\big(c_1^2+c_2^2\big)} \,\,\,<\, \infty \,.$$
 Let $V\in W^{1,2}(A_1\cup A_2,\R^\dxd)$ with $dV=0$ and let $R_1$ and $R_2$ be rotations associated to the restriction of $V$ to $A_1$ and $A_2$, respectively. 
 In the following calculation, we first use that $R_1-R_2$ is constant. Then we apply the inequality $(a+b)^2\le2(a^2+b^2)$ and the fact that the $L^2$-norm on increasing sets increases. Finally we plug in the assumptions. This yields
 \begin{eqnarray*}
  \|R_2-R_1\|_{L^2(A_2)}^2 
  &=& \la(A_2) |R_2-R_1|^2 = \frac{\la(A_2)}{\la(A_1\cap A_2)} \|R_2-R_1\|_{L^2(A_1\cap A_2)}^2  \\
  &\le& \frac{\la(A_2)}{\la(A_1\cap A_2)} \cdot 2 \big(\|R_2-V\|_{L^2(A_1\cap A_2)}^2+\|V-R_1\|_{L^2(A_1\cap A_2)}^2\big) \\
  &\le& \frac{2\la(A_2)}{\la(A_1\cap A_2)} \big(\|R_2-V\|_{L^2(A_2)}^2+\|V-R_1\|_{L^2(A_1)}^2\big) \\
  &\le& \frac{2\la(A_2)}{\la(A_1\cap A_2)} \big(c_2^2\|\dist(V,\sod)\|_{L^2(A_2)}^2+c_1^2\|\dist(V,\sod)\|_{L^2(A_1)}^2\big) \\
  &\le& \frac{2\la(A_2)}{\la(A_1\cap A_2)} \big(c_1^2+c_2^2\big) \|\dist(V,\sod)\|_{L^2(A_1\cup A_2)}^2 \,.
 \end{eqnarray*}
 We set $R=R_1$ and estimate using again elementary inequalities, the assumptions and finally the just obtained estimate of $ \|R_2-R_1\|_{L^2(A_2)}$
 \begin{eqnarray*}
  \|V-R_1\|_{L^2(A_1\cup A_2)}^2 
  &\le& \|V-R_1\|_{L^2(A_1)}^2 + \|V-R_1\|_{L^2(A_2)}^2 \\
  &\le& \|V-R_1\|_{L^2(A_1)}^2 + 2\big(\|V-R_2\|_{L^2(A_2)}^2 +\|R_2-R_1\|_{L^2(A_2)}^2\big)\\
  &\le& c_1^2\|\dist(V,\sod)\|_{L^2(A_1)}^2 + 2c_2^2\|\dist(V,\sod)\|_{L^2(A_2)}^2 \\
  &&    \;+\, 2\|R_2-R_1\|_{L^2(A_2)}^2\\
  &\le& 2\big(c_1^2+c_2^2\big) \|\dist(V,\sod)\|_{L^2(A_1\cup A_2)}^2 + 2\|R_2-R_1\|_{L^2(A_2)}^2\\
  &\le& \Big(\frac{4\la(A_2)}{\la(A_1\cap A_2)}+2\Big) \big(c_1^2+c_2^2\big) \|\dist(V,\sod)\|_{L^2(A_1\cup A_2)}^2 \,,
 \end{eqnarray*}
 which proves the lemma. 
\end{proof}

The case $dV=0$ of Theorem~\ref{thm:rigidity} is preponed into the following lemma. It looks almost like the rigidity estimate of Friesecke et al., but it handles closed 1-forms. In contrast, \cite[Theorem~3.1]{fjm2} considers only exact 1-forms.

\begin{lem} \label{lem:dVgleich0}
 Let $d\ge2$ and $M\subset\R^d$ be open, connected and bounded with Lipschitz boundary. Then there exists a constant $C(M)$ such that for all $V\in W^{1,2}(M,\R^\dxd)$ with $dV=0$ there exists a rotation $R\in \sod$ with
 $$ \|V-R\|_{L^2(M)} \,\le\, C(M) \|\dist(V,\sod)\|_{L^2(M)} \,.$$
\end{lem}
\begin{proof}
 We show this lemma by a covering argument. For $x\in \overline{M}$, let $A_x\subseteq\overline{M}$ be a contractible open neighbourhood of $x$ in $\overline{M}$. Since $\overline{M}$ is compact, there exists a finite subcover of $(A_x)_{x\in \overline{M}}$ of $\overline{M}$. Since $M$ is connected, we can arrange the subcover $A_1,\ldots,A_K$ such that $A_k\cap\bigcup_{l=1}^{k-1}A_l \ne\emptyset$ for all $k\in\{2,\ldots,K\}$. These sets have positive Lebesgue measure. Moreover, for each $k\in\{1,\ldots,K\}$, there is an open set $B_k$ such that $B_k\subseteq A_k\subseteq \overline{B_k}$. Let $C_k=C(B_k)$ be the constant in the rigidity estimate \cite[Theorem 3.1]{fjm2} of Friesecke, James and M\"uller associated to $B_k$. Note that it does not matter whether we use $A_k$ or $B_k$ in their rigidity estimate. 
 
 Let $V\in W^{1,2}(M,\R^\dxd)$ with $dV=0$. Let $k\in\{1,\ldots,K\}$. Since $A_k$ is contractible, there exist $v_k\in W^{2,2}(A_k,\R^d)$ with $V=Dv_k$ on $A_k$. Of course, the functions $v_k$ need not fit together to a global function $v$. Nevertheless, for each $k\in\{1,\ldots,K\}$, there exists a rotation $R_k\in \sod$ such that
 $$ \|Dv_k-R_k\|_{L^2(A_k)} \le C_k \|\dist(Dv_k,\sod)\|_{L^2(A_k)} \,.$$
 
 Using Lemma~\ref{lem:U1cupU2}, we show by induction on $k$, that there exist constants $\tilde{C}_k$ (independent of $V$) and rotations $\tilde{R}_k$ such that
 $$ \|V-\tilde{R}_k\|_{L^2(\bigcup_{l=1}^kA_l)} \le \tilde{C}_k \|\dist(V,\sod)\|_{L^2(\bigcup_{l=1}^kA_l)} $$
 for all $k\in\{1,\ldots,K\}$, which implies the theorem since $\bigcup_{l=1}^KA_l=\overline{M}$. 
\end{proof}

Now we are ready to prove the main rigidity estimate.

\begin{proof}[Proof of Theorem~\ref{thm:rigidity}]
 The case $d=2$ and $p=1$ is already covered by M\"uller et.\ al.\ in \cite[Theorem~3.3]{msz13}. Therefore we may assume $p>1$.
 
 First we prove the theorem for $V\in W^{1,p}(M,\R^\dxd)$. We claim that this implies $V\in L^2(M,\R^\dxd)$. Indeed, $M$ is bounded, and if $2d/(2+d)\le p\le2$ then $1\le p\le 2\le dp/(d-1p)$. Therefore Sobolev's Lemma (see \cite[Theorem 1.3.3(b)]{s95}, for instance) states that
 \begin{equation} \label{eq:sobolev}
  \|V\|_{L^2(M)} \,=\,  \|V\|_{W^{0,2}(M)} \,\le\, \Cdrei \|V\|_{W^{1,p}(M)} 
 \end{equation}
 for some constant $\Cdrei=\Cdrei(M,p)>0$.

 Let $i\in\{1,\ldots,d\}$. Considering the $i$th line $V_i$ as a 1-form, we look for $1$-forms $W_i$ which solve of the equation
 $$ dW_i = dV_i $$
 Obviously, $W_i=V_i$ is a solution. Moreover, $dV_i\in W^{0,p}\Om^2(\overline{M})$, which is the space of $2$-forms with coefficients in $W^{0,p}(\overline{M})$. According to Lemma~3.2.1 of \cite{s95} we choose a solution $W_i\in W^{1,p}\Om^1(\overline{M})$ such that
 \begin{equation}
  \label{eq:estimate}
  \|W_i\|_{W^{1,p}\Om^1(\overline{M})} \le \Cvier \|dV_i\|_{W^{0,p}\Om^2(\overline{M})}
 \end{equation}
 for some constant $\Cvier=\Cvier(M,p)>0$. Note that \cite[Lemma 3.2.1]{s95} requires $p>1$. Therefore this was assumed in the beginning of the proof. Since this lemma is stated for compact $\partial$-manifolds\footnote{A $\partial$-manifold is a complete manifold with boundary equipped with an oriented smooth atlas, see \cite[Definition 1.1.2]{s95}}, we worked on $\overline{M}$. Note that $\overline{M}$ is a compact $\partial$-manifold since $M$ is open and bounded with smooth boundary. 

 Now we define $U_i:=V_i-W_i$. Then $dU_i=dV_i-dW_i=0$. We set $W=(W_i)_{1\le i\le d}$, $U=(U_i)_{1\le i\le d}$. 
 By Lemma~\ref{lem:dVgleich0}, there exist a constant $\Ceins$, only depending on $M$, and a rotation $R\in \sod$ such that
 $$ \|U - R\|_{L^2(M)} \le \Ceins \|\dist(U,\sod)\|_{L^2(M)} \,.$$
 Using the triangle inequality twice and in between the assertion just above, we estimate
 \begin{eqnarray*}
  \|V-R\|_{L^2(M)}
  &=&   \|W+U-R\|_{L^2(M)} \\
  &\le& \|U-R\|_{L^2(M)} + \|W\|_{L^2(M)} \\
  &\le& \Ceins\|\dist(U,\sod)\|_{L^2(M)} + \|W\|_{L^2(M)} \\
  &=&   \Ceins\|\dist(V-W,\sod)\|_{L^2(M)} + \|W\|_{L^2(M)} \\
  &\le& \Ceins\|\dist(V,\sod)\|_{L^2(M)} + (\Ceins+1) \|W\|_{L^2(M)} 
 \end{eqnarray*} 
 Combining estimate \eqref{eq:sobolev} for $W$, i.e.\ Sobolev's Lemma, and estimate \eqref{eq:estimate} yields
 $$   \|W\|_{L^2(M)} \,\le\, \Cdrei \|W\|_{W^{1,p}(M)} \,\le\, \Cdrei\Cvier \|dV\|_{W^{0,p}(\overline{M})} = \Cdrei\Cvier \|dV\|_{L^p(M)} \,.$$
 By setting $\Czwei=(\Ceins+1)\Cdrei\Cvier$, we arrive at
 $$ \|V-R\|_{L^2(M)} \,\le\,
    \Ceins \|\dist(V,\sod)\|_{L^2(M)} +
    \Czwei \|dV\|_{L^p(M)} \,.$$
 which proves the theorem in the case $V\in W^{1,p}(M,\R^\dxd)$.
 
 For general $V\in L^2(M,\R^\dxd)$ with $dV\in L^p(M)$, we use a sequence $V^m\in C^\infty(M,\R^\dxd)$, $m\in\N$, which converges point-wise almost everywhere and with $\|V-V^m\|_{L^2(M)}\to0$ and $\|dV-dV^m\|_{L^p(M)}\to0$ as $m\to\infty$. Then also $\|\dist(V,\sod)-\dist(V^m,\sod)\|_{L^2(M)}\to0$ and the theorem follows. 
\end{proof}

\begin{proof}[Proof of Corollary~\ref{cor:rigidity}]
 Let $v_1,\ldots,v_d\in\R^d$ be vectors such that $ [M] = \R^d\big/\{z_1v_1+\ldots+z_dv_d\mid z_1,\ldots,z_d\in\Z\} $
 and define $M:=\{\la_1v_1+\dots+\la_dv_d\mid \la_1,\ldots,\la_d\in[0,1)\}$. We choose a ball $B\subseteq\R^d$ such that $B\supseteq M$. Moreover, let $\widetilde{M}$ be the union of $n$ translated copies of $M$ such that $\widetilde{M}\supseteq B$ (with some suitable $n\in\N$). We identify any function on $[M]$ with the function on $M$ evaluated at the corresponding representatives and extend it periodically to $\widetilde{M}$. Applying Theorem~\ref{thm:rigidity} to the ball $B$ yields
 \begin{eqnarray*}
  \|V-R\|_{L^2([M])}
  &\le& \|V-R\|_{L^2(B)} \\
  &\le& \Ceins(B) \|\dist(V,\sod)\|_{L^2(B)} +  \Czwei(B,p) \|dV\|_{L^p(B)} \\
  &\le& \Ceins(B) \|\dist(V,\sod)\|_{L^2(\widetilde{M})} +  \Czwei(B,p) \|dV\|_{L^p(\widetilde{M})} \\
  &=& \sqrt{n}\Ceins(B) \|\dist(V,\sod)\|_{L^2([M])} +  \sqrt[p]{n}\Czwei(B,p) \|dV\|_{L^p([M])} \,,
 \end{eqnarray*}
 where we used $M\subseteq B\subseteq \widetilde{M}$ and the facts that all functions are periodically extended to $\widetilde{M}$ and that $\widetilde{M}$ consists of $n$ copies of $M$. Therefore the corollary follows with $\Ceins([M])=\sqrt{n}\Ceins(B)$ and $\Czwei([M],p)=\sqrt[p]{n}\Czwei(B,p)$.
\end{proof}

Finally we proof the behaviour of the constants under scaling.
\begin{proof}[Proof of Lemma~\ref{lem:scale}]
 Let $\wti{M}:=\e M$ be the scaled domain. Let $\wti{V}\in L^2(\wti{M},\R^\dxd)$ with $d\wti{V}\in L^p(\wti{M})$. We define $V\in L^2(M,\R^\dxd)$ by $V(x):=\wti{V}(\e x)$, $x\in M$. 
 
 A change of variables yields
 $$ \int_M |V(x)-R|^2\,dx = \int_{M} |\wti{V}(\e x)-R|^2\,dx = \e^{-d}\int_{\wti{M}}|\wti{V}(y)-R|^2\,dy  $$
 and therefore
 $$ \|V-R\|_{L^2(M)} =  \e^{-\frac{d}{2}}\|\wti{V}-R\|_{L^2(\wti{M})} \,.$$
 Analogously,
 $$ \|\dist(V,\sod)\|_{L^2(M)} =  \e^{-\frac{d}{2}}\|\dist(\wti{V},\sod)\|_{L^2(\wti{M})} \,.$$
 Moreover, $dV(x)=\e\, d\wti{V}(\e x)$ and thus
 $$ \int_{M}|dV(x)|^p\,dx = \int_{M} \e^{p}\,|d\wti{V}(\e x)|^p\,dx = \e^{p-d} \int_{\wti{M}} |d\wti{V}(y)|^p\,dy\,,$$
 which implies $dV\in L^p(M)$ and
 $$ \|dV\|_{L^p(M)}  = \e^{1-\frac{d}{p}} \|d\wti{V}\|_{L^p(\wti{M})}\,.$$
 Using Theorem~\ref{thm:rigidity} on $M$, we conclude
 \begin{eqnarray*}
  \|\wti{V}-R\|_{L^2(\wti{M})}
  &=&
  \e^{\frac{d}{2}} \|V-R\|_{L^2(M)} \\
  &\le&
  \e^{\frac{d}{2}} \Ceins(M) \|\dist(V,\sod)\|_{L^2(M)} + \e^{\frac{d}{2}} \Czwei(M,p)  \|dV\|_{L^p(M)}\\
  &=&
  \Ceins(M) \|\dist(\wti{V},\sod)\|_{L^2(\wti{M})} +  \Czwei(M,p) \e^{\frac{d}{2}+1-\frac{d}{p}} \|d\wti{V}\|_{L^p(\wti{M})}\,.
 \end{eqnarray*}
 Since $\wti{V}$ was arbitrary, we can choose $\Ceins(\e M) = \Ceins(M)$ and $\Czwei(\e M,p)= \e^{\frac{d}{2}-\frac{d}{p}+1} \Czwei(M,p)$, as desired. The proof for the torus is analogous.
\end{proof}

\section{Spontaneous Rotational Symmetry Breaking} \label{cha:symmetrybreaking}

Let us start with an informal description of the crystal. The crystal is given by random points in a box $\I$, which are the centres of the molecules. Thus there is no reference lattice. We assume that the crystal has a favourite structure which should be interpreted as a property of the considered material. This structure is given by a fixed tessellation of $\R^d$. The random points $\P$ determine a set $\T$ of tiles such that each tile in $\T$ is an enlarged $\ep$-perturbation of a standard tile and such that $\T$ locally looks like the given tessellation. The perturbed tiles need not cover the whole box $\I$. The remaining ``holes'' are the defects. Almost all defects are feasible. We only require that each defect has a minimum size, i.e.\ the boundary of a defect does not come closer than $3\rho$ to itself (for some fixed $\rho\in(0,1)$). But the defects may be arbitrarily large and may also have Burgers vectors. Thus there may exist edge, screw and also mixed dislocations. We assume that the crystal is connected and sufficiently large, i.e.\ its size is comparable to the size of the box.  

The distribution of the points is given in the Gibbsian setting using a Poisson Point Process as reference measure. The Hamiltonian consists of three parts. The first part is given by some local Hamiltonians which measures the energy costs due to local deformations of the crystal. These local Hamiltonians are part of the model and shall fulfil a reasonable inequality. They can be given by a pair-potential using adjacent points, for instance (cf.\ Section~\ref{sec:concretemodels}). The second part can be interpreted as a surface energy. It punishes defects proportional to their surface. The last part of the Hamiltonian can be thought as a chemical potential; increasing it favours more points. Then we show that, in an appropriate limit, the local deformation of the crystal is close to a constant rotation.

The organisation of this chapter is as follows. In Section~\ref{sec:modeldef} we define the model in detail. After an overview we describe first the tessellation and then the crystal. Thereafter, we define the local deformation of the crystal as well as the Hamiltonian and the corresponding probability measure. Then we state the main theorem in Section~\ref{sec:thmstate}, which will be proved in Section~\ref{sec:thmproof}. The structure of the proof is explained in the beginning of that section. Finally, we give two examples of concrete models in Section~\ref{sec:concretemodels}.

\subsection{Definition of the Model} \label{sec:modeldef}
First we outline the components of our model.
\begin{compactenum}[\hspace{1.5em}1.]
 \item A periodic locally finite tessellation of $\R^d$, whose tiles are closed polytopes (maybe of different types).
 \item A parameter $\ep>0$, which measures the size of the allowed deformation of the crystal.
 \item A parameter $\rho\in(0,\rho_\text{max})$, which is a lower bound of the size of a defect.
 \item A constant $\co>0$, which is a relative lower bound on the number of the tiles of the crystal.
 \item Some local Hamiltonians, which measure the local deformation of a tile, and constants $\ci>0$, $\cii\in\R$ satisfying a certain inequality (cf.\ \eqref{eq:conHloc} below).
 \item A function $S$, which measures the surface of the defects, and a constant $\ciii>0$ satisfying a certain condition (cf.\ \eqref{eq:conS} below).
\end{compactenum}
In the following subsections, we describe the model accurately. 

\subsubsection{The Underlying Tessellation}

We choose a tessellation $\M$ of the space $\R^d$, $d\ge2$, with the following properties. Each tile $\st\in\M$ is a closed polytope. There are finitely many different types $i\in I$ of tiles. If two tiles have the same type, then their geometric shape and size as well as the types and the placement of their neighbouring tiles are identical. We allow different tile types since they naturally arise if one considers a densest sphere packing in dimension $d\ge3$, for instance. The tessellation shall be locally finite and $B_0$-periodic for a finite box $B_0$ which is the image of the cube $[0,1]^d$ under some linear map $L$. Thus the vectors $Le_j$, $j=1,\ldots,d$, span the box $B_0$ (where $e_j$ denotes the $j$th unit vector). 

Throughout we fix some $\ep>0$, $\rho\in(0,\rho_\text{max})$ and $\co>0$, where $\rho_\text{max}:=1\wedge\min\{\dist(\st,\ti{\st})\mid\st,\ti{\st}\in\M,\st\cap\ti{\st}=\emptyset\}/3$.

For each $i\in I$, we choose a fixed tile of type $i$ in $B_0$, which we denote by $\st^i$. Denoting its corners by $s_1,\ldots,s_{n_i}$, we define the set
$$ \Uept \,:=\, \big\{ \t=\hull\{x_1,\ldots,x_{n_i}\}\mid x_l\in\R^d \text{ s.th.\ }|x_l-s_l|\le\ep, 1\le l\le n_i, \,\wedge\, \la(\t)\ge\la(\st^i)\big\} $$
of all enlarged perturbed tiles. Moreover, we define the ``special Euclidean group'' $\SEept$ of $\Uept$ by
$$ \SEept \,:=\, \big\{a+R\cdot\t\mid a\in\R^d,\, R\in\sod,\, \t\in\Uept\big\} \,.$$
In the following, a ``standard'' tile (as in $\M$) is denoted by $\st$, while a perturbed tile is denoted by $\t$. Moreover, if $\T$ is a set of tiles, we define $\UT:=\{x\in\R^d\mid\exi\t\in\T:x\in\t\}$.

\subsubsection{The Crystal} \label{ssec:crystal}

Let $N\in\N$. Let the torus
$$\I := \R^d\Big/\big\{N(z_1Le_1+\ldots+z_dLe_d) \mid z_1,\ldots,z_d\in\Z\big\} $$
be the ``universe'' of the crystal, with periodic boundary conditions. Moreover, let $\ti{\Om},\F,\mu$ be a suitable probability space and for $\om\in\ti{\Om}$ let 
$$ \P = \P(\om) = \{X_1,\ldots,X_{|\P|}\} \subset \I $$
be Poissonian points, which shall model the centres of the molecules of the crystal; this means that $X_1, X_2,\ldots $ is a sequence of iid random variables which are uniformly distributed on $\I$ and independent of $|\P|$, and $\mu(|\P|=k)=e^{-\la(\I)}\la(\I)^k/k!$, $k\in\N_0$. Note that we suppress the $N$-dependency of $\ti{\Om}$ and $\P$ (and of $\Om$ and $\T$ defined later) to simplify the notation as $N$ is clear from the context.

The molecules of the crystal shall compose a perturbation of the tessellation which may have all kinds of defects.
We will define the set $\T=\T(\om)$ of perturbed tiles. The following construction is a bit complicated, but has the advantage that an upcoming condition is quite simple; the condition ensures that a point configuration is admitted. First we define a set $\hat{\T}_{\text{psbl}}$ which contains all possibly perturbed tiles whose corners are taken from the point configuration. Here we do not impose any condition on the relative locations of the perturbed tiles to each other. But we do impose such conditions in the next step, in which we define when a subset $\ti{\T}\subseteq\hat{\T}_{\text{psbl}}$ is called a locally $\M$-like set of tiles: locally, the relative locations of the tiles must be such as in $\M$. Finally we define a particular  locally $\M$-like set of tiles $\T$, which is the set containing all perturbed tiles of the crystal. It is a maximal locally $\M$-like set of tiles under the conditions that it is connected and that the tiles are not too close to each other (at the boundary of the defects).

\begin{figure}
\begin{center} \includegraphics[width=1\textwidth]{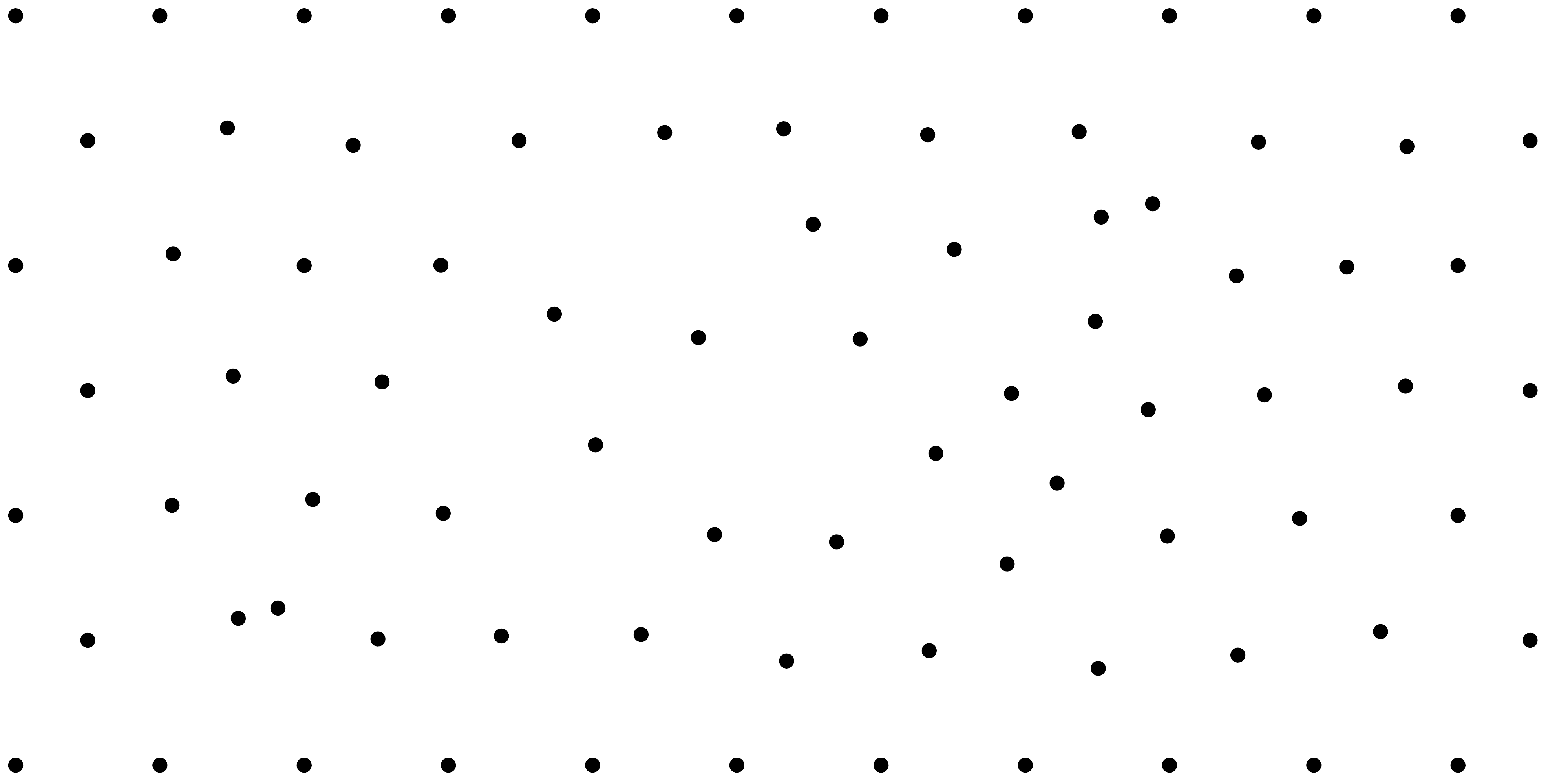} \end{center}
\caption{A random point configuration (with periodic boundary condition) \label{fig:config}}
\end{figure}
\begin{figure}
\begin{center} \includegraphics[width=1\textwidth]{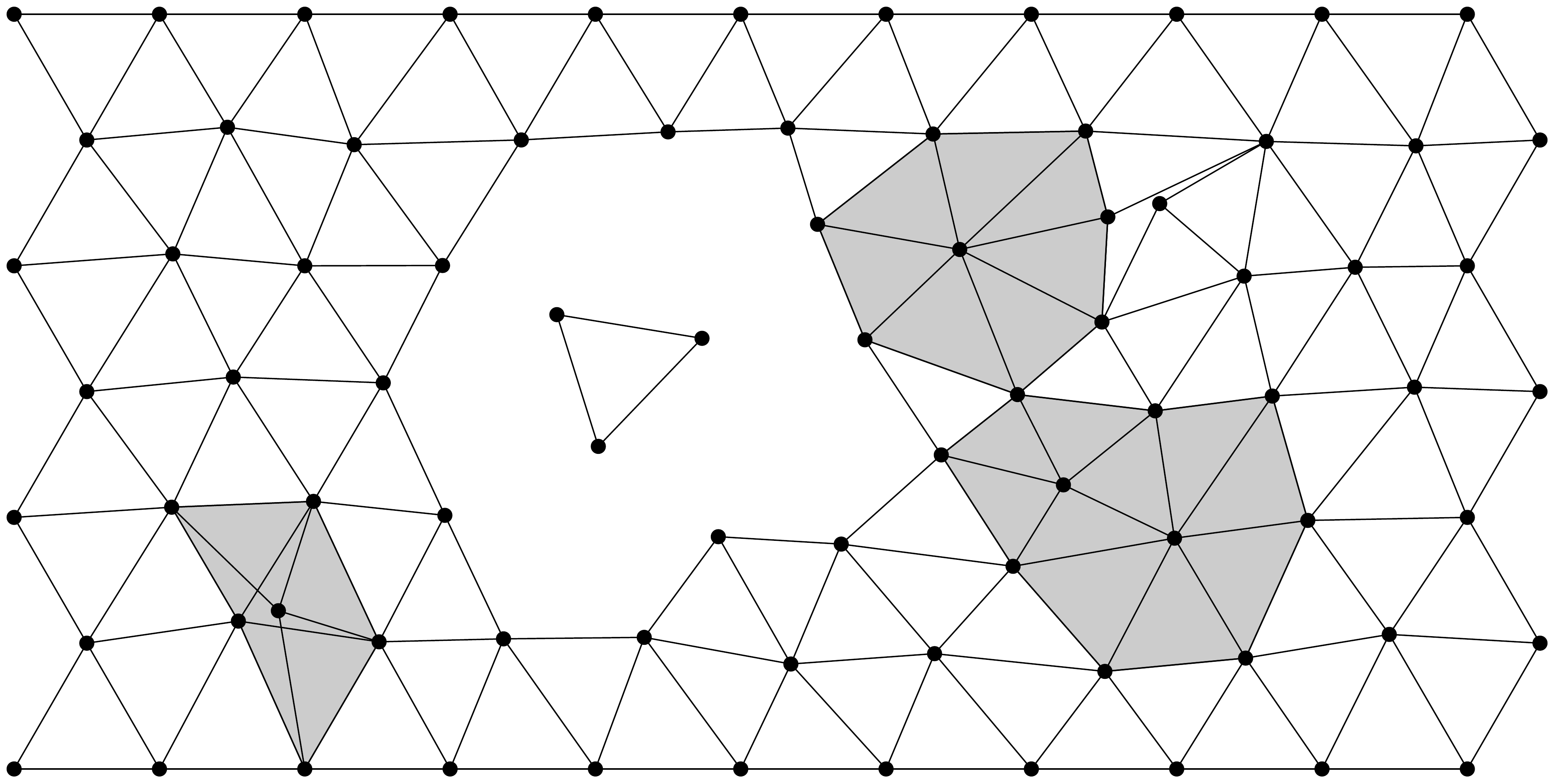} \end{center}
\caption{The set $\hat{\T}_{\text{psbl}}$ of perturbed triangles which does not look like $\M$ in the grey shaded regions \label{fig:Tdach}}
\end{figure}
\begin{figure}
\begin{center} \includegraphics[width=1\textwidth]{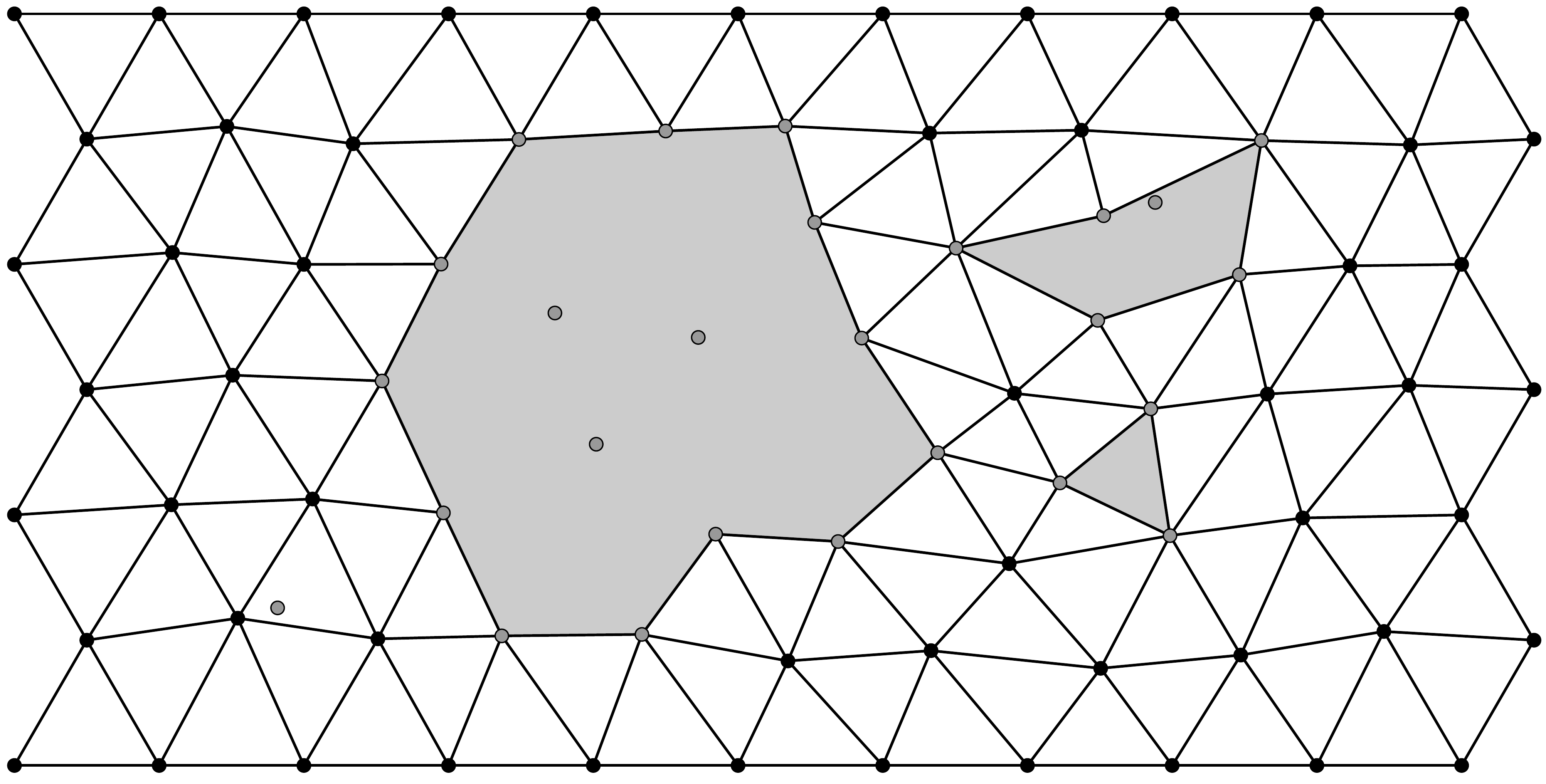} \end{center}
\caption{The crystal $\T$ with defects grey shaded and the surface points $\del\P$ in dark grey\label{fig:T}}
\end{figure} 
Before stating the precise definitions, we give an example. The underlying tessellation is just the two-dimensional triangular lattice. We start with a random point configuration (with periodic boundary condition) which is illustrated in Figure~\ref{fig:config}. Then the set $\hat{\T}_{\text{psbl}}$ of all possibly perturbed tiles contains all triangles (3 points connected by lines) in Figure~\ref{fig:Tdach}, regardless whether they are white or grey shaded. Note that there is a quadrilateral just right to the upper right grey shaded area. It is not included in $\hat{\T}_{\text{psbl}}$ as it is not a triangle. Similarly, the big 13-gon is not included, but the triangle inside is. In the grey shaded regions, $\hat{\T}_{\text{psbl}}$ does not look like the triangular lattice since the triangles do overlap or there is an interior vertex with five or seven adjacent triangles. Therefore, we have to omit some triangles in the grey shaded regions in order to get a locally $\M$-like set of tiles. Finally, the crystal $\T$ is drawn in Figure~\ref{fig:T}. It contains all white triangles. The grey shaded regions (including the grey triangle) are the defects of the crystal.  The triangle formed by the three points inside the huge defect is not included since the crystal must be connected. Furthermore, $\hat{\T}_{\text{psbl}}$ contains two triangles using the point inside the upper right defect. But they are not included in $\T$ since otherwise some triangles would be too close to each other. 

Now we state the precise definitions. Let
$$
 \hat{\T}_{\text{psbl}} \,:=\, \big\{ \t=\hull\{X_{j_1},\ldots,X_{j_k}\} \mid \{j_1,\ldots,j_k\} \subset \{1,\ldots,|\P|\}, \exi i\in I: \t\in\SEept \big\}
$$
be the set of all possibly perturbed tiles. Any subset $\ti{\T}\subset\hat{\T}_{\text{psbl}}$ is called a locally $\M$-like set of tiles, if for $j=1,\ldots,j_{\ti{\T}}$ (with some $j_{\ti{\T}}\in\N$), there are are sets $\ti{\T}_j\subset\hat{\T}_{\text{psbl}}$, $\M_j\subset\M$ and continuous bijective maps 
$$v_j:\Un{\ti{\T}_j}\to\Un{\M_j}$$
mapping each tile $\t\in \SEept$ to a tile $a+\st^i\in\M_j$ (with some $\ti{a}\in\R^d$) such that
$$ \ti{\T} = \bigcup_{j=1}^{j_{\ti{\T}}} \ti{\T}_j $$
and the sets $\ti{\T}_j$ do overlap, i.e.\ all intersections $\Un{\ti{\T}_j}\cap\Un{\ti{\T}_{j'}}$ consist only of whole tiles if they are not empty. 
Thus $\ti{\T}$ is a locally $\M$-like set of tiles if it looks locally like $\M$. Now we define $\T=\T(\om)$ to be a largest subset of $\hat{\T}_{\text{psbl}}$ such that
\begin{compactenum}[\hspace{1.5em}(i)]
 \item $\T$ is a locally $\M$-like set of tiles,
 \item $\UT$ is connected,
 \item if $\t\cap\ti{\t}=\emptyset$ then even $\dist(\t,\ti{\t})>3\rho$ holds for all $\t,\ti{\t}\in\T$ and
 \item \label{en:slit} for all $\t\in\T$, all faces $F$ of $\t$ and for all $\ti{\t}\in\T$ with $F\nsubseteq\ti{\t}$ there exists a point $x\in F$ such that $\dist(x,\ti{\t})>3\rho$. 
\end{compactenum}
Here ``a largest subset'' is understood as a subset whose cardinality (number of tiles) is maximal under all subsets with these properties. In fact, there need not exist a unique largest subset. In that case, we choose one of them according to some fixed rule.

A tile $\t\in\T$ inherits its type from the corresponding tile in $\M$ using the bijections introduced above. We denote it by $\typt$.

Furthermore, we define the set of surface points of $\P$ as follows:
\begin{equation} \label{eq:surfacepoints}
 \del\P \,:=\, \big\{ x\in\P \mid x\in\del\UT\text{ or }x\notin\V(\T)\big\} \,,
\end{equation}
where $\del\UT$ denotes the topological boundary of the set $\UT$ and $\V(\T)$ is the set of points of $\P$, which are vertices of any tile $\t\in\T$. In the example above, the surface points are drawn in grey in Figure~\ref{fig:T}. Note that there are surface points which are not vertices of any tile. We will call such surface points also exterior points (though they can also lie inside the crystal, as one of them does in the example). Such points are possible, but will be unlikely.

We need only one condition on the set $\P$. We namely require that the crystal has a minimum size. Thereto we define the space of admitted configuration to be 
$$ \Om:=\{\om\in\ti{\Om}\mid |\T|\ge \co N^d\} \,.$$
Then $\Om\ne\emptyset$ for large enough $N$ (even for all $N\in\N$ if $\co\le1$) as restricting $\M$ to $\I$ yields an allowed point configuration. Thereto we had to choose $\rho<\rho_\text{max}\le\min\{\dist(\st,\ti{\st})\mid\st,\ti{\st}\in\M,\st\cap\ti{\st}=\emptyset\}/3$. Otherwise even the points of $\M$ would not compose a huge crystal. 

Note that we do not require a minimal distance between two points and that there may exist points inside a tile which do not belong to the tile. But all such points are included in the surface points $\del\P$, which consists not only of the surface vertices of $\T$, but also of the points not belonging to any tile.

\subsubsection{The Local Deformation of the Crystal}

Now we define a random function $V=V(\om)\in L^2(\UT,\R^\dxd)$ which measures the local deformation (rotation and scaling) of the crystal. Thereto, for $i\in I$, we partition the tile $\st^i$ into simplices $\st^{i,1},\ldots,\st^{i,J_i}$. For any $\t\in \SEept$ we define the bijective map
\begin{equation} \label{eq:defvt}
 v_{\scriptscriptstyle\t} : \t \to \st^i  
\end{equation}
such that its restriction to $v_{\scriptscriptstyle\t}^{-1}[\st^{i,j}]$ is affine linear for each $j\in\{1,\ldots,J_i\}$. Using these maps, we define
$$ V : \UT \to \R^\dxd,\quad x\mapsto \nabla v_{\scriptscriptstyle\t} (x)\quad\text{if }x\in\t\,.$$
Note that the Jacobi matrix $\nabla v_{\scriptscriptstyle\t}$ is not well-defined on the boundary of the pre-image of a simplex; but since these boundaries have zero Lebesgue measure, this is irrelevant. Then $V$ is a piecewise constant function on $\UT$. Though it is locally defined as a derivative, it is, in general,  globally not a derivative, since there may be defects with Burgers vectors.

\subsubsection{The Hamiltonian}

We assume that some local Hamiltonians
$$ \hloc^i: \Uept \to \R\,,\quad i\in I\,, $$
are given which are continuous and fulfil
\begin{align} 
 &\exi \ci>0 \exi \cii\in\R\, \fa i\in I \fa \t\in\Uept: \notag\\
 &\qquad \hloc^i(\t)-\hloc^i(\st^i)\,\ge\, \ci \|\dist(\nabla v_{\scriptscriptstyle\t},\sod)\|^2_{L^2(\t)} + \cii\big(\la(\t)-\la(\st^i)\big)\,.\label{eq:conHloc}
\end{align}
A tile $\t\in\T$ satisfies $\t\in\SEept[\typts]$. Therefore $\t=a+R\cdot\ti{\t}$ for some $a\in\R^d$, $R\in\sod$ and $\ti{\t}\in\Uept[\typts]$ . If several choices of $a$, $R$ and $\ti{\t}$ are possible, we choose one of them according to some fixed rule. We extend $\hloc^i$, $i\in I$, to $\T$ by setting $\hloc^\typts(\t):=\hloc^\typts(\ti{\t})$.

Let further a quantity $S:\Om\to\R$ be given which measures the number of surface points of the crystal in the following sense:
\begin{equation} \label{eq:conS}
 \exi \ciii>0\fa N\in\N \fa \om\in\Om:\quad \ciii |\del\P| \lle S \quad\text{and}\quad \del\P=\emptyset\,\Rightarrow\,S=0 \,.
\end{equation}

Now we define the Hamiltonian
\begin{equation} \label{eq:defH}
 H_\sn(\om)\,:=\, \sum_{\t\in\T} \hloc^{\typts}(\t) + \s S - m |\P| 
\end{equation}
for $\s>0$, $m\in\R$ and $N\in\N$. The first addend measures the local energy of the crystal caused by the perturbation of $\M$. The term $\s S$ represents the surface energy. Finally, $m$ can be interpreted as a chemical potential. Using this Hamiltonian we define for $\be>0$, $\s>0$, $m\in\R$ and $N\in\N$ the partition sum
\begin{equation} \label{eq:defZ}
 Z_\bsn \,:=\, \int_\Om e^{-\be H_\sn} \,d\mu
\end{equation}
and the probability measure $P_\bsn$ via
\begin{equation} \label{eq:defP}
 \frac{dP_\bsn}{d\mu} \,:=\, \frac{1}{Z_\bsn} e^{-\be H_\sn} \,.
\end{equation}
Let $E_\bsn$ denote the expectation with respect to $P_\bsn$.

Note that $P_\bsn$ is well-defined as $Z_\bsn\in(0,\infty)$, at least for large enough $N$. Indeed, the lower bound on $H_\sn$ provided by Lemma~\ref{lem:dist-estimate} below implies $Z_\bsn<\infty$ (cf.\ the remark after that lemma). Furthermore, Lemma~\ref{lem:partitionsum} below implies $Z_\bsn>0$ for large enough $N$.

\subsection{The Main Result} \label{sec:thmstate}

Now we are ready to state the main result.
\begin{thm} \label{thm:symmetrybreaking}
 There exist $m_0\in\R$ and constants $\cxxvii,\cxxviii>0$ and $\cxxix\in\R$ depending only on the model, but not on $m$, $\s$, $\be$ or $N$, such that the rotational symmetry of the crystal is broken in the following sense: 
 $$\forall\, m\ge m_0:\quad\adjustlimits\lim_{\be\to\infty} \limsup_{N\to\infty} \sup_{\s\ge\s_0(N,m)} 
  E_\bsn\bigg[ \inf_{R\in\sod} \frac{1}{|\T|}\sum_{\t\in\T} \|V-R\|^2_{L^2(\t)}\bigg] \;=\; 0 $$
 where $\s_0(N,m):=\cxxvii N^2 + \cxxviii m + \cxxix $.
\end{thm}

The main constraint of this theorem is that the estimate is not uniform in the size of the box since $\s_0$ depends on $N$. Thus it does not carry over to infinite-volume limits. The reason for that $N$-dependency lies in the scaling behaviour of the constants in Theorem~\ref{thm:rigidity} as stated in Lemma~\ref{lem:scale}. It is not possible to get better results using the chosen method.

Another constraint is that we assumed or rather conditioned on the event that the size of the crystal is comparable to the box size, i.e.\ $|\T|\ge\co N^d$. Whether this event has large probability is a different topic and not discussed in this article. But one might expect that its probability is large if the chemical potential $m$ is large enough. Then more points are more likely and they should form more tiles, since otherwise they are surface points which are punished with $\s\ge m$.

Let us further remark, that the crystal consists only of enlarged perturbed tiles, i.e.\ the Lebesgue measure of any perturbed tile must not be smaller than the Lebesgue measure of the corresponding standard tile. Therefore, it is not possible to cover the whole box with more tiles than the standard tessellation would need. This may be considered as a hard-core condition. Furthermore, the whole perturbed tile must be $\ep$-close to a standard tile. For instance, the postulate that only the edge lengths are close to the corresponding standard edge lengths might not be enough. 

Moreover, we assume in the definition of $\T$ that each defect has a minimum size: non-adjacent tiles must have distance larger than $3\rho$. This condition is crucial to extend $V$ into the defects. 

We also assume by definition that the crystal is connected.  This assumption is necessary. If the crystal consists of two components, for example, there is no reason why one could use the same rotation $R$ for both components. Indeed, the second component could be a rotated copy of the first one.

Finally, we equipped the box $\I$ with periodic boundary conditions. This has in particular the advantage that configurations without defects have no boundary, which is a technical relaxation, especially in Lemma~\ref{lem:partitionsum}. Otherwise, the periodic boundary is not essentially used.

Despite these constraints, especially the non-uniformity in $N$, Theorem~\ref{thm:symmetrybreaking} has the feature that it handles almost all kinds of defects, including unbounded and dislocation defects. Up to the author's knowledge, it is the first result on spontaneous symmetry breaking allowing such general defects.

\subsection{Proof of the Main Result} \label{sec:thmproof}

Before we start the proof, we give an overview. Generally, we prove Theorem~\ref{thm:symmetrybreaking} using more or less the same approach as Heydenreich, Merkl and Rolles used in \cite{hmr13}. But the implementation of that approach is different.

One main difference is that we work directly on the level of the derivatives: Indeed $V$ is matrix-valued and locally the derivative of a function $v_{\scriptscriptstyle\t}$. But globally, $V$ need not be any derivative. Moreover, $v_{\scriptscriptstyle\t}$ is the inverse of the corresponding function in \cite{hmr13}. This is due to the fact that there is no reference lattice.

First we extend the function $V$ into the defects in Subsection~\ref{sec:extensionV}. Thereto we use a tube-neighbourhood of $\UT$. This extension is different to the extension in \cite{hmr13} since we consider different kinds of defects. In Subsection~\ref{sec:cardinality} we define the standard configuration and estimate the cardinality of some subsets of $\P$ and $\T$; this section has no counterpart in \cite{hmr13}. Afterwards, in Subsection~\ref{sec:estimateH}, we prove an estimate for the Hamiltonian, which is an analogue to \cite[Lemma~3.2]{hmr13}. Though its proof is different, it uses the same general idea, namely to apply a rigidity estimate. In Subsection~\ref{sec:partitionsum} a lower bound for the partition sum is proven, which is used in Subsection~\ref{sec:internalenergy} to receive an upper bound for the internal energy. The proofs of these results, which are analogues to \cite[Lemma~3.1]{hmr13} and \cite[Lemma~3.3]{hmr13}, respectively, use ideas from their proofs. Finally, in Subsection~\ref{sec:results}, we prove a corollary which states the main result in different forms and also implies Theorem~\ref{thm:symmetrybreaking}.

In the following we need quite a lot different constants. Unless explicitly stated, they are all uniform constants. Almost all of them depend on the model, i.e.\ on the tessellation, the local Hamiltonians, the surface measure $S$ or on the constants $\ep,\rho,\co,\ci,\cii,\ciii$. But they are independent of $m$, $\s$, $\be$, $N$ and $\om$.

The constants in the lemmas and in the proofs are numbered separately. The constants in the lemmas are needed globally. Though we need the constants in the proofs only locally, they are numbered in ascending order to avoid confusion. Most of the constants are positive, but some can be any real number. In that case the constant has a little ${}^{\scriptscriptstyle\R}$ as superscript.

\subsubsection{Extension into the Defects} \label{sec:extensionV}
First we want to extend the random function $V=V(\om)\in L^2(\UT,\R^\dxd)$, which measures the local deformation of the crystal, into the defects. We receive a random function also denoted by $V=V(\om)$ with $V\in L^2(\I,\R^\dxd)$ and $dV\in L^p(\I)$, $p\ge1$. For a set $A\subseteq\I$, let $A^c:=\I\setminus A$ denote the complement of $A$ in $\I$.

We define a $\rho$-tube-neigh\-bour\-hood $\del^{\overline{0\rho}}\UT$ of $\UT$ using a homeomorphism
$$ g=(g_\del,g_t):\, \del^{\overline{0\rho}}\UT \,\to\, \del\UT\times[0,1]$$
such that $\del^{\overline{0\rho}}\UT \subseteq (\interior\UT)^c$, $g(x)=(x,0)$ for all $x\in\del\UT$ and such that $dg_t$ exists and is uniformly bounded. Though not formally required, one can imagine $\del^{\overline{0\rho}}\UT$ as the set of points whose distance from $\UT$ is at most $\rho$. Then $g$ is some parametrisation of this set. This is also the reason for the notation. The proof of the existence of such a homeomorphism is given in Lemma~\ref{lem:exist_g} below. The main ingredient is a vector field $w$ defined on $\del\UT$, which exists since the distance of two disjoint tiles is greater than $3\rho$ by the definition of $\T$.

This construction is schematically drawn in Figure~\ref{fig:extdefects}.
\begin{figure}
 \begin{center}
  \begin{overpic}[width=0.85\textwidth]{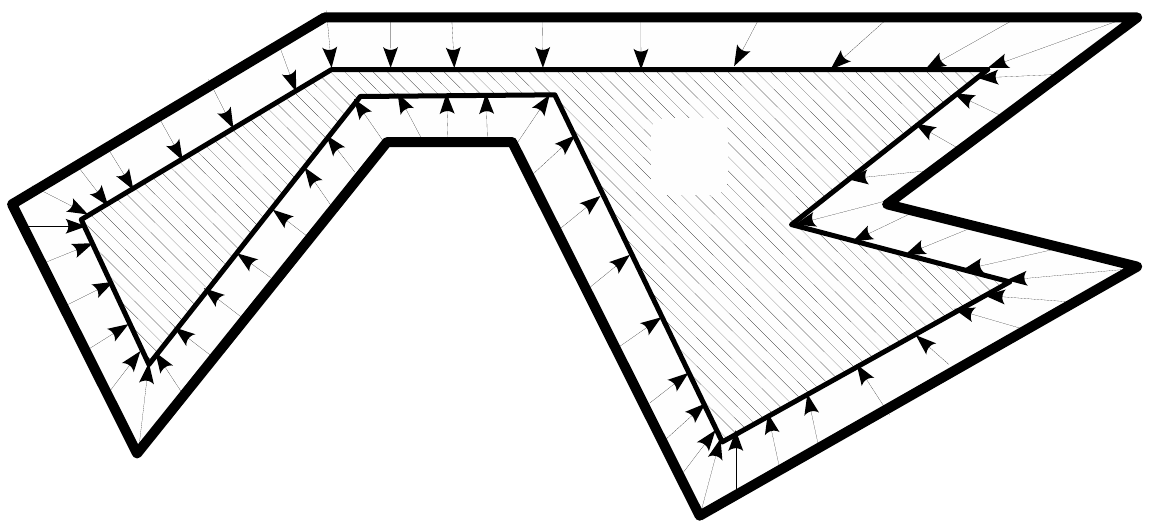} 
   \put(58.3,30.5){\Large $\wti{R}$}
   \put(36,18){\Large $V$}
  \end{overpic}
 \end{center}
 \caption{A defect (with arrows or hatched) of the crystal (white, outside), the $\rho$-tube-neigh\-bour\-hood (with arrows) and the vector field $w$ (the arrows) \label{fig:extdefects}}
\end{figure}
The crystal is the white area outside and the defect consists of the hatched area and of the area with arrows. The latter one is the $\rho$-tube-neigh\-bour\-hood $\del^{\overline{0\rho}}\UT$. We will extend the function $V$, which is already defined in the white area, into the defects by setting it constant inside the hatched area and by interpolating inside the area with arrows.

In order to extend $V$, we choose a rotation $\wti{R}=\wti{R}(\om)\in\sod$ uniformly at random, independently of $\P$. We could also use a fixed rotation; but if it is chosen uniformly at random, the random variable $V$ is rotational invariant. Moreover, let $\ti{V}^n:\UT\to\R^\dxd$, $n\in\N$, be smooth functions which converge to $V$ on $\UT$. First we extend $\ti{V}^n$ to $V^n$ as follows:
$$ V^n(x) \,:=\, \begin{cases}
                  \ti{V}^n(x)					&\text{if }x\in\UT \\
                  (1-g_t(x))\ti{V}^n(g_\del(x)) + g_t(x) \wti{R}	&\text{if }x\in\del^{\overline{0\rho}}\UT \\
                  \wti{R}	&\text{if }x\in(\UT)^c\cap(\del^{\overline{0\rho}}\UT)^c  \,.
                 \end{cases}$$
Finally, we define $V$ as the $L^2$-limit of $V^n$.
This limit exists and is independent of the choice of the sequence $\ti{V}^n$. Moreover, Lemma~\ref{lem:interpolate} below implies that $dV\in L^p(\I)$, $p\ge1$.

Now we prove the existence of the homeomorphism $g$.

\begin{lem} \label{lem:exist_g}
 There exists a constant $\cxxvi>0$ such that for all $N\in\N$ and $\om\in\Om$, there exists a Lipschitz-continuous homeomorphism
 $$ g=(g_\del,g_t):\, \del^{\overline{0\rho}}\UT \,\to\, \del\UT\times[0,1]$$
 with Lipschitz-continuous inverse such that first $\del^{\overline{0\rho}}\UT \subseteq (\interior\UT)^c $, second $g(x)=(x,0)$ for all $x\in\del\UT$ and finally $dg_t$ exists with $ |dg_t| \le \cxxvi$.
\end{lem}
\begin{proof}
 For any $z\in\R^d\setminus0$, we can decompose a vector $w\in\R^d$ into
 $$ w = w_{\perp z} + w_{\parallel z} \,$$
 where $w_{\parallel z}$ is the orthogonal projection of $w$ onto $z\R$ and $w_{\perp z}:=w-w_{\parallel z}$. This decomposition is linear in $w$. 
 
 In order to construct the homeomorphism, we will define a vector field $w:\del\UT\to\R^d$. The boundary of $\UT$ is Lipschitz as it consists of $(d-1)$-dimensional polytopes. Thus there exist open sets $W_j\subset\R^d$ covering $\del\UT$, open sets $\ti{U}_j\subset\R^{d-1}$ and compatible Lipschitz continuous bijective maps $h_j:(-2,2)\times\ti{U}_j\to W_j$ mapping $\{0\}\times\ti{U}_j$ to $\del\UT$, $(-2,0)\times\ti{U}_j$ to $\interior(\UT)$ and $(0,2)\times\ti{U}_j$ to $(\UT)^c$, $j\in J$. We can further assume that for all $x,y\in\del\UT$ with $|x-y|\le2\rho$, there exists $j\in J$ with $x,y\in W_j$, because $|x-y|\le2\rho$ implies that $x$ and $y$ belong to the same tile or to adjacent tiles (the distance of non-adjacent tiles is greater than $3\rho$ by the definition of $\T$). Note that the angles between two adjacent polytopes are uniformly bounded away from zero. Indeed, if the defect is locally due to a missing tile, this follows from the fact that all tiles are $\ep$-perturbations of the given tessellation; and if the defect is locally an inserted wedge, i.e.\ it comes locally from a slit, then the angle of that wedge is bounded away from zero by condition (\ref{en:slit}) in the definition of $\T$. Therefore the Lipschitz constants of $(h_j)_{j\in J}$ can be uniformly bounded for all $N\in\N$ and $\om\in\Om$.
 
 We define the vector field $\ti{w}:\del\UT\to\R^d$ by pushing the field $u(z)=(1,0,\ldots,0)$, $z\in\{0\}\times\ti{U}_j$ forward with $h_j$, i.e.\ $\ti{w}(x):=h_j[(1,0,\ldots,0)+h_j^{-1}(x)]-x$ for suitable $j$, $x\in\del\UT$. Then $\ti{w}$ is uniformly Lipschitz and $|\ti{w}|$ is uniformly bounded away from zero and infinity (in $\om$ and $x$). Now we scale $\ti{w}$ to lower its Lipschitz constant and size. This yields a vector field $w:\del\UT\to\R^d$ such that for all $x,y\in\del\UT$:
 \begin{compactenum}[\hspace{1.5em}(i)]
  \item $x+tw(x)\notin\UT$ for all $t\in(0,1]$
  \item $|w(x)-w(y)|\le \clii |x-y|$
  \item $\cliii\le|w(x)|\le \rho$
  \item $|w(y)_{\perp(y-x)}|\ge |w(y)|/\cli$ if $0<|x-y|\le2\rho$.
 \end{compactenum}
 for some universal constants $\cli,\clii,\cliii>0$ satisfying
 \begin{equation} \label{eq:lipschitzconstants}
  (1+\cli)\clii < 1 \qquad\text{and}\qquad \rho + \sqrt{2}\clii < 1\,.
 \end{equation}
 Condition (iv), which is scale-invariant, already holds for $\ti{w}$: since $|x-y|\le2\rho$ implies $x,y\in U_j$ for some $j$, we can use $u(h^{-1}_j(y))=(1,0,\ldots,0)\perp h_j^{-1}(x)-h_j^{-1}(y)$ and the Lipschitz property of $h_j$ to derive (iv). Conditions (iii) and (ii) and Equation \eqref{eq:lipschitzconstants} are be fulfilled by scaling ($\cli$ and $\rho<\rho_\text{max}\le1$ are already fixed). Condition (i) follows from (iii) since the distance between two disjoint tiles is greater than $3\rho$ by the definition of $\T$.
 
 Using the vector field $w$, we define the function 
 \begin{eqnarray*}
  f:\, \del\UT\times[0,1] &\to& \I \\
       (x,t) &\mapsto& x + tw(x)\,,
 \end{eqnarray*}
 which will be the inverse of the homeomorphism $g$. 
 It is Lipschitz-continuous since
 $$ \big| f(x,t)-f(y,s)\big|
   \,=\, \big|(x-y)+t(w(x)-w(y))+(t-s)w(y)\big|
   \lle  (1+\clii)|x-y| + \rho |t-s| $$
 by properties (ii) and (iii) of $w$.
 
 We will also derive a reverse Lipschitz condition to conclude that $f$ is injective and its inverse is also Lipschitz-continuous. Thereto let $x,y\in\del\UT$ and $t,s\in[0,1]$. First we assume $x\ne y$. We estimate using the triangle inequality and the Lipschitz continuity of $w$
 \begin{eqnarray}
  \big|x-y+(t-s)w(y)\big|
  &\le&
  \big|x+tw(x)-y-sw(y)\big|+t\big|w(y)-w(x)\big| \notag\\
  &\le&
  \big|f(x,t)-f(y,s)\big|+\clii|x-y|\,. \label{eq:g3}
 \end{eqnarray}
 Pythagoras' Theorem yields that
 \begin{eqnarray}
  \big|x-y+(t-s)w(y)\big|^2
  &=&
  \big|x-y+(t-s)w(y)_{\parallel(x-y)}\big|^2 + \big|(t-s)w(y)_{\perp(x-y)}\big|^2 \notag\\
  &\ge&
  \big((1-\rho)|x-y|\big)^2 + \big|(t-s)w(y)_{\perp(x-y)}\big|^2 \label{eq:g4}
 \end{eqnarray}
 since $|(t-s)w(y)_{\parallel(x-y)}|\le|w(y)|\le\rho$. 
 
 The inequality  $\sqrt{2}\sqrt{a^2+b^2}\ge(a+b)$ yields (\ref{eq:g4}) without the squares, but with an additional $\sqrt{2}$ on the left hand side. Combing this with (\ref{eq:g3}) yields
 \begin{eqnarray}
  \sqrt{2}\big|f(x,t)-f(y,s)\big|
  &\ge&
  (1-\rho)|x-y| +\big|(t-s)w(y)_{\perp(x-y)}\big| - \sqrt{2}\clii|x-y| \notag\\
  &=&
  \big(1-(\rho+\sqrt{2}\clii)\big)\big|x-y\big|+ \big|w(y)_{\perp(x-y)}\big| \big|t-s\big|\,. \label{eq:g5}
 \end{eqnarray}
 Note that  $\rho+\sqrt{2}\clii<1$ by (\ref{eq:lipschitzconstants}).
 
 Now if $|x-y|\le2\rho$, then $|w(y)_{\perp(x-y)}|\ge|w(y)|/\cli\ge\cliii/\cli$. Otherwise $|w(y)_{\perp(x-y)}|\ge0$ and $\tfrac12|x-y|\ge\rho\ge\rho|t-s|$. Therefore, in both cases (\ref{eq:g5}) implies
 \begin{equation} \label{eq:revLipf}
  \big| f(x,t)-f(y,s)\big| \gge \cliv\big(|x-y|+|t-s|\big)
 \end{equation}
 for some constant $\cliv>0$. Now we consider the case $x=y$. Then
 $$  \big|f(x,t)-f(y,s)\big| \,=\, \big|x+tw(x)-y-sw(y)\big| \,=\, |t-s||w(x)| \ge \cliii|t-s| $$
 by property (iii). Thus \eqref{eq:revLipf} also holds in that case.
 
 Inequality \eqref{eq:revLipf} implies that $f$ is indeed injective. Moreover, property (i) implies 
 $ \del^{\overline{0\rho}}\UT := \image f \subseteq (\interior\UT)^c $.
 We define
 $$ g : \, \del^{\overline{0\rho}}\UT \,\to\, \del\UT\times[0,1]\,, \quad z\mapsto f^{-1}(z) $$
 as the inverse of $f$. Then $g(x)=(x,0)$ for all $x\in\UT$ holds by definition. Furthermore, \eqref{eq:revLipf} implies that $g$ is Lipschitz continuous. Thus the existence of $dg_t$ as well as the bound $|dg_t|\le\cxxvi$ for some $\cxxvi>0$ follow. 
\end{proof}

Finally in this section, we prove a bound of $\dist(V,\sod)$ and $dV$.
\begin{lem} \label{lem:interpolate}
 There exists a constant $\cv>0$ such that for all $N\in\N$ and $\om\in\Om$ and $\la$-almost all $x\in\I$
 $$  \dist(V(x),\sod)^2 \lle \cv \quad\text{and}\quad |dV(x)|\lle \cv\,.$$ 
\end{lem}
\begin{proof}
 First we note that $V$ and $V^n$ are uniformly bounded on $\UT$ and therefore also on $\del\UT$ since any tile $\t\in\T$ is, up to translation and rotation, $\ep$-close to $\st^\typts$. Moreover, $\wti{R}$ is uniformly bounded since $\sod$ is compact. Thus $V^n$ and therefore $V$ is uniformly bounded on the whole $\I$, which implies the first inequality.
 
 For the second inequality, we first note that since $V\!\!\upharpoonright_\UT$ is locally the derivative of a continuous piecewise affine linear function, we could also choose $\ti{V}^n$ locally as a derivative. Therefore $dV=0$ on $\UT$. Moreover, $dV=0$ on $(\UT)^c\cap(\del^{\overline{0\rho}}\UT)^c$ since $\wti{R}$ is constant. Finally, we calculate for $x\in\del^{\overline{0\rho}}\UT$
 \begin{eqnarray*}
  dV^n(x) &=&
  (1-g_t(x))d\ti{V}^n(g_\del(x)) - dg_t(x)\wedge \ti{V}^n(g_\del(x)) + g_t(x) d\wti{R} + dg_t(x)\wedge \wti{R} \\
  &=&
  -dg_t(x)\wedge\big(V^n(g_\del(x))-\wti{R}\big)
 \end{eqnarray*}
 since $d\ti{V}^n=0$ on $\UT$ since $\ti{V}^n$ is locally a derivative. Since $|dg_t|\le\cxxvi$ by Lemma~\ref{lem:exist_g}, since $V^n$ and $\wti{R}$ are uniformly bounded and since $V^n\to V$, the second inequality follows.
\end{proof}

\subsubsection{Cardinality of Subsets of $\P$ and $\T$} \label{sec:cardinality}

In this section, we give some definitions and some lemmas, which estimate the cardinality of several subsets of $\P$ and $\T$.

First we define the \emph{standard configuration} $\sta\in\Om$ with points $\Q$ and tiles $\U$ as a fixed element of $\Om$ such that the crystal is exactly the tessellation $\M$. More precisely, using the notation $\V(\M\!\!\upharpoonright_{\I})$ for the vertices of $\M$ inside $\I$, we require
$$ \Q := \P(\sta) = \V(\M\!\!\upharpoonright_{\I})\quad\text{ and thus }\quad \U := \T(\sta) = \M\!\!\upharpoonright_{\I} \,. $$
The choice of $\rho<\rho_\text{max}$ ensures the last equation and $\sta\in\Om$ (if $N$ is large enough, depending on $\co$).

We will need some subsets of $\T$ and $\P$. We define the set of boundary tiles by
$$ \del\T \,:=\, \big\{\t\in\T: \t\cap\del\UT\ne \emptyset\big\} $$
and for $i\in I$ the set 
$$ \T^i \,:=\, \big\{\t\in\T: \typ{\t}=i\big\}\,,$$ 
which consists of all tiles of type $i$ (recall that $\typt$ denotes the type of $\t$). Obviously, $\U^i$ denotes the set of all tiles of type $i$ in $\U$, $i\in I$. Let us further recall that we already defined the surface points $\del\P$ in \eqref{eq:surfacepoints} as follows:
$$ \del\P \,:=\, \big\{ x\in\P \mid x\in\del\UT\text{ or }x\notin\V(\T)\big\} \,,$$
where $\del\UT$ denotes the topological boundary and $\V(\T)$ is the set of points of $\P$, which are vertices of any tile $\t\in\T$. Furthermore, we need the notation
$$ \Pe \,:=\, \big\{x\in\P\mid x\notin\V(\T)\big\} $$
for the exterior points. Note that the exterior points, which are not contained in any perturbed tile of $\T$, are contained in the set of surface points. Note further that the standard configuration has empty boundary, i.e.\ $\del\Q=\emptyset$ and $\del\U=\emptyset$.

These sets are illustrated in Figure~\ref{fig:subsetsPT}. It shows the example of a crystal used in Section~\ref{ssec:crystal}. The defects are shaded in dark grey. The boundary tiles are light grey shaded. All surface points are drawn in grey. The five surface points which also are exterior points are marked with a circle. 
\begin{figure}
 \begin{center}
  \includegraphics[width=1\textwidth]{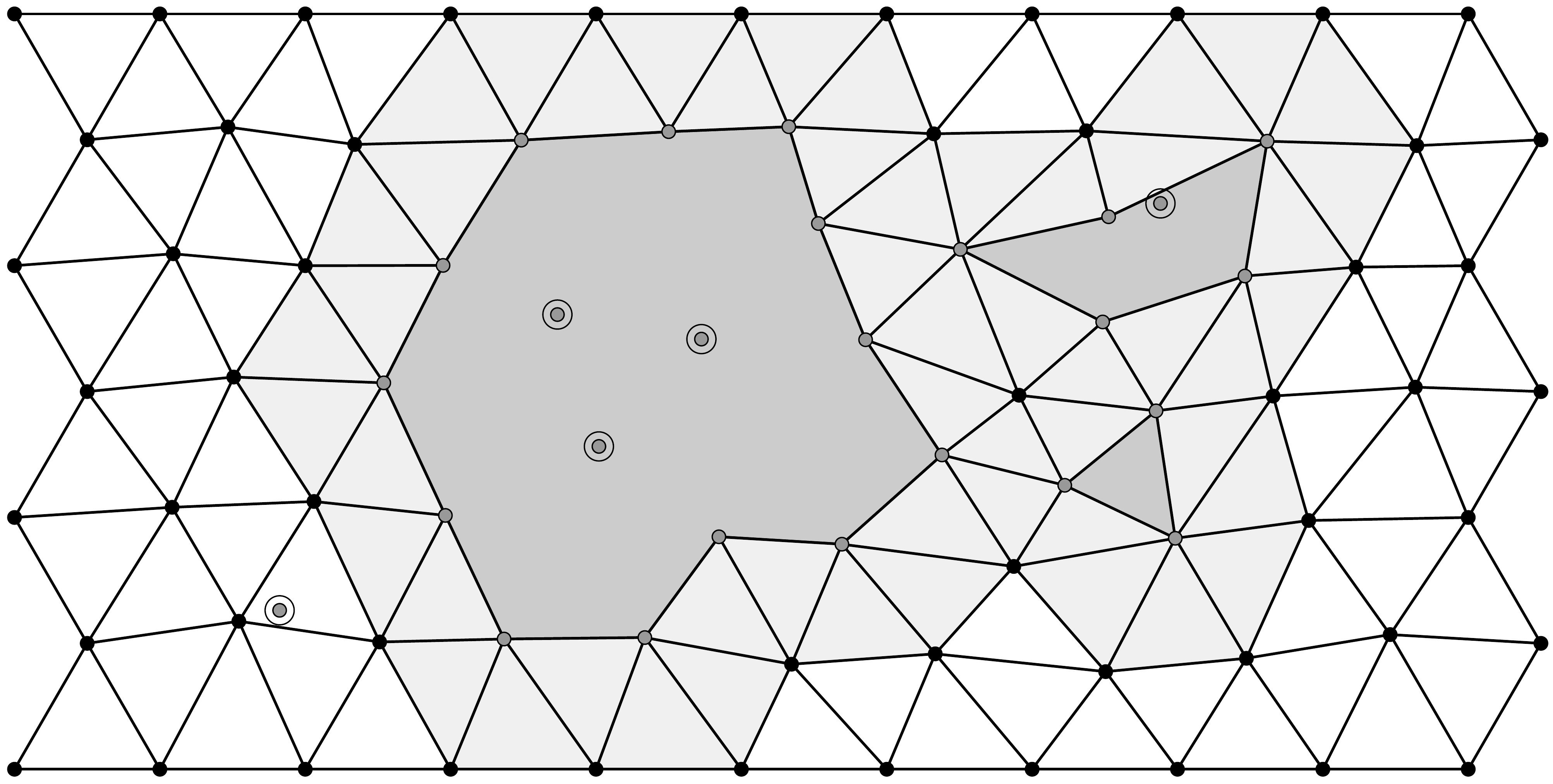}
 \end{center}
 \caption{The boundary tiles (light grey), the surface points (grey) and the exterior points (with circle) of a crystal with defects (dark grey area) \label{fig:subsetsPT}} 
\end{figure}
Note that one of the exterior points is inside the crystal but is not a vertex of any tile.

Similarly to the tile types, we may also partition the vertices $\V(\M)$ of $\M$ into types $j\in J$, depending on their adjacent tiles. The assignment of the types to the tiles and vertices shall in particular imply that, for all $i,l\in I, j\in J$ the quantities
\begin{equation} \label{eq:bef}
b_{i,l} := \sum_{\substack{\ti{\st}\in\M \\ \typ{\ti{\st}}=l}} \1_{\st^i\cap\ti{\st}\ne\emptyset} 
\,,\qquad
e_{ij} := \sum_{\substack{x\in\V(\M)\\ \typv{x}=j}} \1_{x\in\st^i}
\,,\qquad
f_{ij} := \sum_{\substack{\st\in\M\\ \typ{\st}=i}} \1_{x^j\in\st}\,
\end{equation}
are well-defined, finite and independent of the choice of $\st^i$ of type $i$ and $x^j\in\V(\M)$ of type $j$, respectively. These quantities are interpreted as follows: $b_{i,l}$ denotes the number of neighbouring tiles of type $l$ to a tile of type $i$, and $e_{ij}$ denotes the number of adjacent vertices of type $j$ to a tile of type $i$, and finally $f_{ij}$ denotes the number of adjacent tiles of type $i$ to a vertex of type $j$. 

In fact, we need the different vertex types only in this section; therefore the letter $j$ may denote various index variables later. But the letter $i$ will only be used for a tile type.

The following lemma shows that the number of tiles of type $i$ is bounded by the number of such tiles in the standard configuration, up to an error in terms of the number of boundary tiles.
\begin{lem} \label{lem:estTUdelT}
 There exists a constant $\cvi>0$ such that for all $N\in\N$, $\om\in\Om$ and $i\in I$ the following inequality holds:
 $$ |\T^i| \lle |\U^i| + \cvi \,|\del\T| \,.$$
\end{lem}
\begin{proof}
 First we show that there exist constants $c_{i,l}>0$, $i,l\in I$, and $\cvi>0$ such that for all $N\in\N$ and $\om\in\Om$
 \begin{equation} \label{eq:cilTi-Tl}
  \big|c_{i,l}|\T^l|-|\T^i|\big| \lle \cvi \,|\del\T|\,.
 \end{equation}
 Let $i,l\in I$. We define the quantity
 $$ A:= \adjustlimits\sum_{\t\in\T^i}\sum_{\tit\in\T^l} \1_{\t\cap\tit\ne\emptyset} \,.$$
 By the definition of $b_{i,l}$ in equation (\ref{eq:bef}), it follows that, for all $\t\in\T^i$,
 $$ 0 \le \sum_{\tit\in\T^l} \1_{\t\cap\tit\ne\emptyset} \le b_{i,l} 
 \quad\text{ and even }\quad \sum_{\tit\in\T^l} \1_{\t\cap\tit\ne\emptyset} = b_{i,l} 
 \quad\text{if}\,\,\,\, \t\in\T^i\setminus\del\T\,.$$
 Summing over all $\t\in\T^i$ yields
 $$ 0\lle b_{i,l} |\T^i| - A = \sum_{\t\in\T^i}\Big(b_{i,l}-\sum_{\tit\in\T^l} \1_{\t\cap\tit\ne\emptyset}\Big) \lle b_{i,l} |\del\T\cap\T^i| \lle b_{i,l} |\del\T| \,.$$
 Analogously, it follows that
 $$ - b_{l,i}|\del\T| \lle A - b_{l,i}|\T^l| \lle 0 \,.$$
 Adding these two inequalities, we get
 \begin{equation} \label{eq:blidelT}
  - b_{l,i} |\del\T| \lle b_{i,l} |\T^i| - b_{l,i}|\T^l| \lle b_{i,l} |\del\T|
 \end{equation}
 Now we observe that either $b_{i,l}=0=b_{l,i}$ or $b_{i,l}\ne0 \wedge b_{l,i}\ne0$ since $b_{i,l}$ counts the   tiles of type $l$ adjacent to a tile of type $i$. In the latter case, we can define $c_{i,l}:=b_{l,i}/b_{i,l} \in (0,\infty)$ and receive
 \begin{equation} \label{eq:bilne0}
  \big||\T^i| - c_{i,l}|\T^l|\big| \lle \max\{1,c_{i,l}\} \,|\del\T| 
 \end{equation}
 by equation (\ref{eq:blidelT}).
 
 In the general case, there is a sequence $i=i_0,i_1,\ldots,i_n=l$ with some $n\le |I|$ such that $b_{i_{k\!-\!1},i_k}\ne0$ for all $k\in\{1,\ldots,n\}$ since the tessellation $\M$ is connected. Therefore we can define $c_{i,l}:=\prod_{k=1}^nc_{i_{k\!-\!1},i_k}\in(0,\infty)$. Using a telescope sum it follows that
 \begin{eqnarray*}
  \big||\T^i| - c_{i,l}|\T^l|\big| &=&
  \Big||\T^i| - \textstyle\prod\limits_{k=1}^n\displaystyle\!\!c_{i_{k\!-\!1},i_k}\,|\T^l|\Big| 
  = \Big|\sum_{j=1}^n\Big(\textstyle\prod\limits_{k=1}^{j\!-\!1}\displaystyle\!\!c_{i_{k\!-\!1},i_k}\, |\T^{i_{j\!-\!1}}| - \textstyle\prod\limits_{k=1}^{j}\displaystyle\!\!c_{i_{k\!-\!1},i_k}\,|\T^{i_{j}}|\Big)\Big| \\
  &\le&
  \sum_{j=1}^n \textstyle\prod\limits_{k=1}^{j\!-\!1}\displaystyle\!\!c_{i_{k\!-\!1},i_k} \cdot \Big| |\T^{i_{j\!-\!1}}|-c_{i_{j\!-\!1},i_j}|\T^{i_j}|\Big| \\
  &\stackrel{(\ref{eq:bilne0})}{\le}&
  \sum_{j=1}^n\textstyle\prod\limits_{k=1}^{j\!-\!1}\displaystyle\!\!c_{i_{k\!-\!1},i_k}\max\{1,c_{i_{j\!-\!1},i_j}\} \,|\del\T|
  \lle \cvi \,|\del\T|\,,
 \end{eqnarray*}
 where $\cvi$ is the supremum of the last sum over all possible choices of the sequence $i_0,i_1,\ldots,i_n$. For the last line, we used the already covered case $b_{i_{j\!-\!1},i_j}\ne0$. Thus claim \eqref{eq:cilTi-Tl} follows.
 
 Using $\la(\t)\ge\la(\st^i)$ for all $i\in I$ and $\t\in\Uept$, we estimate
 $$
  \sum_{i\in I}|\T^i|\la(\st^i) \,=\, \sum_{\t\in\T} \la(\st^{\typts}) \lle \sum_{\t\in\T} \la(\t) \lle \la(\I) \,=\, \la(\Un{\,\U}) \,=\, \sum_{i\in I} |\U^i| \la(\st^i)
 $$
 since the standard configuration covers the whole box with standard tiles. Therefore there exists $i_0=i_0(\om)\in I$ with $|\T^{i_0}| \le |\U^{i_0}|$.
 
 Let $i\in I$. Using claim \eqref{eq:cilTi-Tl} it follows that
 $$ |\T^i| \lle c_{i,i_0}|\T^{i_0}|+\cvi|\del\T| \lle c_{i,i_0}|\U^{i_0}|+\cvi|\del\T| = |\U^i|+\cvi|\del\T|\,. $$
 The last equality again follows from claim \eqref{eq:cilTi-Tl}, applied to $\U$, since $\del\U=\emptyset$.
\end{proof}

In the next two lemmas, we use the relation $\asymp$ to indicate that the quotient of the left and of the right is uniformly in $N\in\N$ and $\om\in\Om$ bounded away from zero and infinity. But we also state the inequalities we need in the sequel explicitly. First we show that different measurements of the boundary have approximately equal size.
\begin{lem} \label{lem:estBoundary}
 There are constants $\g_i>0$, $i\in I$, such that
 $$ |\del\P\setminus\Pe| \aasymp |\P\setminus\Pe| - \sumigt \aasymp |\del\T| \aasymp \la(\del^{\overline{0\rho}}\UT)\,.$$
 In particular it is shown that there are constants $\cix>0$, $\cx>0$ and $\g_i>0$, $i\in I$, such that for all $N\in\N$ and $\om\in\Om$
 \newcounter{enumisave}
 \begin{enumerate}[\hspace{1.5em}\upshape(a)\hspace{0.5em}]
  \item \label{asa} $|\del\P| \gge |\P| - \sum_{i\in I}\g_i|\T^i|\gge 0$,
  \item \label{asb} $|\del\T| \lle \cix \big(|\P| - \sum_{i\in I}\g_i|\T^i|\big)$ and
  \item \label{asc} $\la(\del^{\overline{0\rho}}\UT) \lle \cx |\del\T|$.
  \setcounter{enumisave}{\value{enumi}}
 \end{enumerate}
\end{lem}
\begin{proof}
  First we show $ |\del\P\setminus\Pe| \asymp |\P\setminus\Pe| - \sum_{i\in I}\g_i|\T^i|$. Thereto, we partition all points $\P$ into the points $\P^j$ of type $j\in J$ and into the exterior points $\Pe$. Of course, a point in $\P\setminus\Pe$ inherits its type from the corresponding point in $\M$. Since  $e_{ij}=\sum_{x\in\P^j}\1_{x\in\t}$ is the number of vertices of type $j$ adjacent to any tile $\t\in\T$ of type $\typt=i$ (including the boundary tiles), see equation (\ref{eq:bef}), it follows that
 \begin{equation} \label{eq:sumxt}
  e_{ij}\,|\T^i| \,=\,\sum_{x\in\P^j}\sum_{\t\in\T^i}\1_{x\in\t}\,. 
 \end{equation}
 We observe that $e_{ij}=0$ iff $f_{ij}=0$ and define
 \begin{equation} \label{eq:defgi}
  \g_i\,:=\,\sum_{j\in J} \frac{1}{|I_j|}\,\frac{e_{ij}}{f_{ij}}\,\1_{f_{ij}\ne0}
 \end{equation}
 with $I_j:=\{i\in I\mid f_{ij}\ne0\}$. Note that $I_j\ne\emptyset$. Therefore
 \begin{align} \label{eq:as2anfang}
   &|\P|-|\Pe|-\sum_{i\in I}\g_i|\T^i| \,=\, \sum_{j\in J}|\P^j|-\sum_{i\in I}\sum_{j\in J} \frac{1}{|I_j|}\,\frac{e_{ij}}{f_{ij}}\,\1_{f_{ij}\ne0}|\T^i| = \sum_{j\in J}\Big(|\P^j|-\sum_{i\in I_j}\frac{1}{|I_j|}\,\frac{e_{ij}}{f_{ij}}|\T^i| \Big)\notag\\
   &\quad=\, \sum_{j\in J}\sum_{i\in I_j}\frac{1}{|I_j|}\Big(|\P^j|-\frac{e_{ij}}{f_{ij}}|\T^i| \Big)
   \stackrel{(\ref{eq:sumxt})}{=}
   \sum_{j\in J}\sum_{i\in I_j}\frac{1}{|I_j|}\sum_{x\in\P^j}
   \Big(1-\frac{1}{f_{ij}}\sum_{\t\in\T^i}\1_{x\in\t}\Big)
 \end{align}
 
 Now we examine the expression
 $$ A_i(x):= 1-\frac{1}{f_{ij}}\sum_{\t\in\T^i}\1_{x\in\t} $$
 for $x\in\P^j$. 
 Since $f_{ij}$ counts  number of tiles of type $i$ adjacent to a vertex of type $j$, it follows that $A_i(x)=0$ if $x\notin\del\P$ and $A_i(x)\le1$ in general. Therefore we can continue (\ref{eq:as2anfang}) as follows:
 $$
 |\P\setminus\Pe|-\sum_{i\in I}\g_i|\T^i| \,\stackrel{\text{\eqref{eq:as2anfang}}}{=}\,
 \adjustlimits\sum_{j\in J}\sum_{i\in I_j}\frac{1}{|I_j|}\sum_{x\in\P^j} A_i(x)
 \lle
 \adjustlimits\sum_{j\in J}\sum_{i\in I_j}\frac{1}{|I_j|}\,|\P^j\cap\del\P| \,=\, |\del\P\setminus\Pe| \,,
 $$
 which is one of the two desired inequalities. By adding $|\Pe|$, this also shows the main inequality of Assertion (\ref{asa}) since $\Pe\subset\del\P$; ``$\ge0$'' follows from (\ref{asb}).
 
 For the other inequality, we define $\P^j_*:=\{x\in\P^j\mid \exi i\in I_j: \sum_{\t\in\T^i}\1_{x\in\t} < f_{ij}\}$. We observe that $x\in\P^j_*$ for some $j$ if a tile is missing which should be adjacent to $x$. Thus $\P^j_*\subseteq\del\P$. If $x\in\del\P\setminus\bigcup_{j}\P^j_*$, then the defect at $x$ is induced by a slit. In that case there is a vertex $y\in\del\P\setminus\Pe$ adjacent to $x$ such that a tile is missing at $y$, i.e.\ $y\in \P^j_*$ for some $j$. Since the vertex degree is uniformly bounded, we conclude
 \begin{equation} \label{eq:PsterngedelP}
  \sum_{j\in J} |\P^j_*| \gge \cxxx |\del\P\setminus\Pe|
 \end{equation}
 for some $\cxxx>0$.
 For $x\in\P^j_*$, let $i_0(x)$ be the smallest $i\in I_j$ with $\sum_{\t\in\T^i}\1_{x\in\t} < f_{ij}$. It follows that
 $$ A_i(x) \ge |I|\cvii \1_{i=i_0(x)} $$
 for $x\in\P^j\cap\del\P$ with
 $ \cvii := \tfrac{1}{|I|}\,\min\big\{\tfrac{1}{f_{ij}}\mid i\in I_j, j\in J\big\} $.
 Plugging this into (\ref{eq:as2anfang}) yields
 \begin{align}  
  |\P\setminus\Pe|-\sum_{i\in I}\g_i|\T^i| &\,=\,
 \adjustlimits\sum_{j\in J}\sum_{i\in I_j}\frac{1}{|I_j|}\sum_{x\in\P^j} A_i(x)
 \gge
 \adjustlimits\sum_{j\in J}\sum_{i\in I_j}\frac{1}{|I_j|}\sum_{x\in\P^j_*}
 |I|\cvii \1_{i=i_0(x)} \notag\\
 &\,=\,\cvii\sum_{j\in J}\frac{|I|}{|I_j|}\adjustlimits\sum_{x\in\P^j_*}\sum_{i\in I_j}\1_{i=i_0(x)}
 =\cvii\sum_{j\in J}\frac{|I|}{|I_j|}\sum_{x\in\P^j_*} 1 \notag\\
 &\gge
 \cvii \sum_{j\in J} |\P^j_*|
 \gge
 \cxxx\cvii \,|\del\P\setminus\Pe|
 \,, \label{eq:delPohnePele}
 \end{align}
 as desired. We used \eqref{eq:PsterngedelP} in the last step.  
 
 Second we show $|\del\P\setminus\Pe| \asymp |\del\T|$. For all $x\in\del\P\setminus\Pe$ there exists at least one $\t\in\del\T$ with $x\in\t$. Therefore
 \begin{eqnarray*}
  |\del\P\setminus\Pe| &\le&
  \adjustlimits\sum_{x\in\del\P\setminus\Pe} \sum_{\t\in\del\T} \1_{x\in\t}
  \,\lle\,
  \adjustlimits\sum_{i\in I}\sum_{\t\in(\T^i\cap\del\T)}\adjustlimits\sum_{j\in J}\sum_{x\in\P^j}\1_{x\in\t} \\
  &\stackrel{(\ref{eq:bef})}{=}& 
  \adjustlimits\sum_{i\in I}\sum_{\t\in(\T^i\cap\del\T)} \sum_{j\in J} e_{ij}
  \,\lle\, |\del\T| \,\max_{i\in I}\big\{\textstyle\sum_{j\in J}\displaystyle e_{ij}\big\} \,.
 \end{eqnarray*} 
 Conversely, for each $\t\in\del\T$, there exists at least one $x\in\del\P\setminus\Pe$ with $x\in\t$. Therefore
 \begin{eqnarray*}
  |\del\T| &\le&
  \adjustlimits\sum_{\t\in\del\T} \sum_{x\in\del\P\setminus\Pe}  \1_{x\in\t}
  \lle\,
  \adjustlimits\sum_{j\in J}\sum_{x\in(\P^j\cap\del\P)}\adjustlimits\sum_{i\in I}\sum_{\t\in\T^i}\1_{x\in\t} \\
  &\stackrel{(\ref{eq:bef})}{=}& 
  \adjustlimits\sum_{j\in J}\sum_{x\in(\P^j\cap\del\P)} \sum_{i\in I} f_{ij}
  \lle |\del\P\setminus\Pe| \,\max_{j\in J}\big\{\textstyle\sum_{i\in I}\displaystyle f_{ij}\big\} 
  \,.
 \end{eqnarray*}
 This and \eqref{eq:delPohnePele} imply Assertion (\ref{asb}) with $\cix:=\max_{j\in J}\sum_{i\in I}f_{ij}/(\cxxx\cvii)$.
 
 Finally, we show $|\del\T| \asymp \la(\del^{\overline{0\rho}}\UT)$. For a set $A\subset\R^d$, let $\OO(A)=\la_{d\!-\!1}(\del A)$ denote the surface area of $A$. Using the Lipschitz continuous homeomorphism $g$, we conclude that  
 \begin{equation*}
   \rho \,\OO(\UT) \aasymp \la(\del^{\overline{0\rho}}\UT) \,.
 \end{equation*}
 Moreover,
 \begin{equation*}
  \OO(\UT) \lle \sum_{\t\in\del\T} \OO(\t) \lle \cxxxii \,|\del\T|
 \end{equation*}
 for some constant $\cxxxii>0$ since the surface area of a tile is uniformly bounded. But for the other direction one has to be careful, since there may exists boundary tiles which do not have a face which is part of $\del\UT$. But let $\del^*\T$ denote the set of boundary tiles having a face which is contained in $\del\UT$. Since for each tile $\t\in\del\T\setminus\del^*\T$ there exists a tile $\tit\in\del^*\T$ with $\t\cap\tit\ne\emptyset$ and since each tile $\tit\in\del^*\T$ intersects at most $\max_{i,l\in I}b_{i,l}$ other tiles, there is a constant $\cxxxiii>0$ such that $|\del^*\T|\ge\cxxxiii|\del\T|$. Since the area of a face of a tile is at least $\cxxxiv>0$ (say), 
 $$ \OO(\UT) \gge \cxxxiv \,|\del^*\T| \gge \cxxxiv\cxxxiii\,|\del\T| $$
 follows. Combining all three displayed formulas in this paragraph yields the claim, which in particular implies Assertion (\ref{asc}). 
\end{proof}

Now we observe that the size of the crystal is comparable to the size of the box, where we can understand each size in two different senses. 

\begin{lem} \label{lem:estTPBoxNd}
 It is true that
 $$ |\T| \aasymp |\P\setminus\Pe| \aasymp N^d \aasymp \la(\I) \,.$$
 In particular, there are constants $\cxvi$, $\cxc$, $\cxiv$, $\cxi$, $\cxv$, $\cxiii>0$ such that for all $N\in\N$ and $\om\in\Om$
 \begin{enumerate}[\hspace{1.5em}\upshape(a)\hspace{0.5em}]
  \setcounter{enumi}{\value{enumisave}}
  \item \label{asd} $\cxvi\,N^d \lle |\P\setminus\Pe| \lle \cxc\,N^d$,
  \item \label{ase} $|\T|\lle \cxiv\,N^d$,
  \item \label{asf} $|\T|\gge \cxi\,\la(\I)\,$ and $\;\la(\I)\gge \cxv\,|\T|$,
  \item \label{asg} $\la(\I)=\cxiii N^d$.  
 \end{enumerate} 
\end{lem}
\begin{proof}
 Since $\I$ consists of $N^d$ copies of the box $B_0$, Assertion (\ref{asg}) follows with $\cxiii=\la(B_0)>0$.
 
 Now note that $\co N^d\le|\T|$ holds by the definition of $\Om$. Moreover,  $\la(\t)\ge\la(\st^\typts)$ for all $\t\in\T$ implies
 $$ \min_{i\in I}\la(\st^i)\,|\T| \lle \sum_{\t\in\T}\la(\t) \lle \la(\I) \,.$$
 Thus we have shown that $|\T| \asymp  N^d \asymp \la(\I)$ as well as Assertions (\ref{ase}) and (\ref{asf}).
 
 Finally, Lemma~\ref{lem:estBoundary} implies  $ |\P\setminus\Pe| \ge \sum_{i\in I}\g_i|\T^i| \ge \min_{i\in I} \g_i\,|\T|$. The other direction follows from the two facts that each point in $\P\setminus\Pe$ is a corner of a tile of $\T$ and that the number of vertices per tile is bounded. This also yields Assertion (\ref{asd}).
\end{proof}

\subsubsection{Estimates for the Hamiltonian} \label{sec:estimateH}

The goal of this subsection is to prove the following estimate for the Hamiltonian, which is an analogue to \cite[Lemma~3.2]{hmr13}.
Thereto we define
\begin{equation} \label{eq:defm0}
 m_0:= \max_{i\in I}\big\{\big(\hloc^i(\st^i)-(\cii-|\cii|)\la(\st^i)\big)/\g^i\big\}\,,
\end{equation}
where the constants $\g_i>0$ depend only on the tessellation and are specified in \eqref{eq:defgi} above.

\begin{lem} \label{lem:H-estimate}
 There exist $\cxxvii,\cxxviii>0$, $\cxxix\in\R$ and $\cxix>0$ such that for all $m\ge m_0$,  $N\in\N$ and  $\s\ge\s_0(N,m)=\cxxvii N^2 + \cxxviii m + \cxxix $ and for all $\om\in\Om$ there exists a random rotation $R=R(\om)\in\sod$ with
 $$ H_\sn(\om) - H_\sn(\sta) \gge \cxix\, \|V-R\|^2_{L^2(\I)} \,.$$
\end{lem}

We partition the proof of Lemma~\ref{lem:H-estimate} into several lemmas. For better readability and shorter formulas, we omit the indexes $\sn$ of $H_\sn$ sometimes in the proofs, but not in the statements of the lemmas.

\begin{lem} \label{lem:dist-estimate}
 There exist constants $\cxvii>0$ and $\cxviii\in\R$ such that for all $m\ge m_0,\s>0, N\in\N$ and $\om\in\Om$ it is true that
 $$ H_\sn(\om)-H_\sn(\sta) 
  \gge
  \ci \|\dist(V,\sod)\|^2_{L^2(\UT)} + \big(\s\ciii-\cxvii m- \cxviii\big)\big(|\P|-\sumigt\big) \,.$$
\end{lem}
\begin{rem*}
 Lemma~\ref{lem:dist-estimate} and Lemma~\ref{lem:estBoundary}(\ref{asa}) imply $H_\sn\ge H_\sn(\sta)+\al|\P|$ with $\al=\min\{\s\ciii-\cxvii m- \cxviii,0\}\le0$. Therefore
 $$ Z_\bsn \lle e^{-\be H_\sn(\sta)}\int_\Om e^{-\be\al|\P|}d\mu \,<\, \infty $$
 since the exponential moment of the Poisson distributed random variable $|\P|$ exists. The conclusion also holds for $m<m_0$ as $H_\sn=H_{\s,m_0,N}-(m-m_0)|\P|$.
\end{rem*}
\begin{proof}
 Using first the definition (\ref{eq:defH}) of $H_\sn$, second the assumption (\ref{eq:conHloc}) on the local Hamiltonians $\hloc^i$ and assumption (\ref{eq:conS}) on the quantity $S$ (note $\del\Q=\emptyset$) and finally Lemma~\ref{lem:estBoundary}(\ref{asa}), we estimate
 \begin{eqnarray}
  \lefteqn{H_\sn(\om)-H_\sn(\sta) \,=\,}\quad  \notag\\
  &=&
  \sum_{\t\in\T}\big(\hloc^{\typts}(\t)-\hloc^{\typts}(\st^\typts)\big) + \sum_{i\in I}|\T^i|\hloc^i(\st^i) - \sum_{i\in I}|\U^i|\hloc^i(\st^i) \notag\\
  && +\, \s S(\om) - \s S(\sta) -m|\P| + m|\Q| \notag\\
  &\ge &
  \sum_{\t\in\T}\Big(\ci\|\dist(\nabla v_{\scriptscriptstyle\t},\sod)\|^2_{L^2(\t)} + \cii\big(\la(\t)-\la(\st^{\typts})\big)\Big) + \sum_{i\in I}\big(|\T^i|-|\U^i|\big)\hloc^i(\st^i) \notag\\
  && +\,\s\ciii|\del\P|-m\big(|\P|-|\Q|\big) \notag\\
  &\ge& 
  \ci \|\dist(V,\sod)\|^2_{L^2(\UT)} + \cii\sum_{\t\in\T}\big(\la(\t)-\la(\st^{\typts})\big) + \sum_{i\in I}\big(|\T^i|-|\U^i|\big)\hloc^i(\st^i) \notag\\
  && +\,\s\ciii\big(|\P|-\sumigt\big) - m\big(|\P|-\sumigt\big) + m\big(|\Q|-\sumigt\big)\,. \label{eq:H-H1}
 \end{eqnarray}
 
 Now we bound the term $\cii\sum_{\t\in\T}\big(\la(\t)-\la(\st^{\typts})\big)$ from below using the fact $\la(\t)\ge\la(\st^\typts)$ for all $\t\in\T$. If $\cii\ge0$, we are done with bounding by $0$, but $\cii<0$ is also possible. The just mentioned fact implies
 $$\la(\UT)+\sum_{i\in I}|\T^i|\la(\st^i) \lle 2 \la(\UT) \lle 2\la(\I) \,=\, 2 \sum_{i\in I}|\U^i|\la(\st^i)\,.$$
 Subtracting $2\sum_{i\in I}|\T^i|\la(\st^i)$ from this inequality yields
 $$ \sum_{\t\in\T}\big(\la(\t)-\la(\st^{\typts})\big) \lle 2 \sum_{i\in I}\big(|\U^i|-|\T^i|\big)\la(\st^i)\,. $$
 Altogether, it follows that
 \begin{equation} \label{eq:DiffVol}
  \cii\sum_{\t\in\T}\big(\la(\t)-\la(\st^{\typts})\big) \gge (\cii\!-\!|\cii|)\sum_{i\in I}\big(|\U^i|-|\T^i|\big)\la(\st^i)
 \end{equation}
 since $\cii-|\cii|=-2|\cii|$ if $\cii<0$ and $\cii-|\cii|=0$ if $\cii\ge0$. 
 
 Moreover, Lemma~\ref{lem:estBoundary}(\ref{asa}) for $\sta$ yields $|\Q|=\sum_{i\in I}\g_i|\U^i|$ since $\del\Q=\emptyset$. Therefore
 \begin{equation} \label{eq:DiffQ}
  m\big(|\Q|-\sumigt\big) \,=\, m\sum_{i\in I}\g_i\big(|\U^i|-|\T^i|\big)\,.
 \end{equation} 
 Plugging (\ref{eq:DiffVol}) and (\ref{eq:DiffQ}) into (\ref{eq:H-H1}) yields
 \begin{eqnarray}
  H_\sn(\om)-H_\sn(\sta) 
  &\ge&
  \ci \|\dist(V,\sod)\|^2_{L^2(\UT)} + (\s\ciii-m)\big(|\P|-\sumigt\big) \notag\\
  && +\, \sum_{i\in I}\big(|\U^i|-|\T^i|\big)\big(m \g_i+(\cii\!-\!|\cii|)\la(\st^i) - \hloc^i(\st^i)\big) \label{eq:H-H2}
 \end{eqnarray}
 
 Since $m \g_i+(\cii\!-\!|\cii|)\la(\st^i) - \hloc^i(\st^i)\ge0$ for $m\ge m_0$ by the choice of $m_0$ in (\ref{eq:defm0}), we can first use Lemma~\ref{lem:estTUdelT} and then Lemma~\ref{lem:estBoundary}(\ref{asb}) and receive 
 \begin{eqnarray*} 
  \lefteqn{\sum_{i\in I}\big(|\U^i|-|\T^i|\big)\big(m \g_i+(\cii\!-\!|\cii|)\la(\st^i) - \hloc^i(\st^i)\big) \,\ge\,}\quad \\
  &\ge&
  -\cvi |\del\T|  \sum_{i\in I}\big(m \g_i+(\cii\!-\!|\cii|)\la(\st^i) - \hloc^i(\st^i)\big)
  \,=\,
  -\cvi |\del\T| (m\cxxxv +\cxxxvi) \\
  &\ge&
  -\cvi\cix\big(|\P|-\sumigt\big)(m\cxxxv +\cxxxvi)
 \end{eqnarray*}
 for constants $\cxxxv>0$ and $\cxxxvi\in\R$ with $m\cxxxv+\cxxxvi>0$ for $m\ge m_0$. 
 Inserting this into (\ref{eq:H-H2}) yields the claim, namely
 $$ H_\sn(\om)-H_\sn(\sta) 
  \ge
  \ci \|\dist(V,\sod)\|^2_{L^2(\UT)} + \big(\s\ciii-\cxvii m- \cxviii\big)\big(|\P|-\sumigt\big) $$
 with constants $\cxvii:=1+\cvi\cix\cxxxv>0$ and $\cxviii:=\cvi\cix\cxxxvi\in\R$.  
\end{proof}

\begin{lem} \label{lem:la-estimate}
 For all $p\in[2d/(2+d),2]$, there exist constants $\cxix>0$ and $\cxx(p)>0$ such that for all $N\in\N$ and $\om\in\Om$ there exists a random rotation $R=R(\om)\in\sod$ with
 $$ c_1\,\|\dist(V,\sod)\|^2_{L^2(\UT)}  \gge \cxix \,\|V-R\|^2_{L^2(\I)} \,-\, \cxx(p)\,N^{2+d-\frac{2d}{p}} \la\big(\del^{\overline{0\rho}}\UT\big)^{\frac2p}\,.$$
\end{lem}
\begin{proof}
 By Corollary~\ref{cor:rigidity} and Lemma~\ref{lem:scale}, there exists a random rotation $R=R(\om)\in\sod$ such that
 \begin{equation} \label{eq:V-Rledist}
  \|V-R\|_{L^2(\I)} \lle \Ceins \|\dist(V,\sod)\|_{L^2(\I)} + N^{\frac d2-\frac dp+1}\Czwei(p) \|dV\|_{L^p(\I)} 
 \end{equation}
 with scale-invariant constants $\Ceins=\Ceins(\Ieins)$ and $\Czwei(p)=\Czwei(\Ieins,p)$.
 
 Since $V=\wti{R}\in\sod$ on $(\UT)^c\cap(\del^{\overline{0\rho}}\UT)^c$ and Lemma~\ref{lem:interpolate}, it follows that
 \begin{eqnarray}
  \|\dist(V,\sod)\|^2_{L^2(\I)}
  &=&
  \|\dist(V,\sod)\|^2_{L^2(\UT)}+\|\dist(V,\sod)\|^2_{L^2(\del^{\overline{0\rho}}\UT)} \notag\\
  &\le&
  \|\dist(V,\sod)\|^2_{L^2(\UT)}+\cv\,\la(\del^{\overline{0\rho}}\UT) \label{eq:distvsod}
 \end{eqnarray}
 and, also using $dV=0$ on $\UT$,
 \begin{equation} \label{eq:dv}
  \|dV\|^2_{L^p(\I)} = \|dV\|^2_{L^p(\del^{\overline{0\rho}}\UT)}
 = \Big(\int_{\del^{\overline{0\rho}}\UT} |dV|^p \,d\la\Big)^{\frac2p}
 \lle \cva^2\la\big(\del^{\overline{0\rho}}\UT\big)^{\frac2p} \,.
 \end{equation}
 Using $\frac2p-1\ge0$ (since $p\le2$) at $*$ yields for all $y\ge0$:
 $$ 
 y=0 \vee y>\rho^d 
 \,\Leftrightarrow\,
 y=0 \vee \rho^{-d}y>1
 \,\stackrel{*}{\Leftrightarrow}\,
 y=0 \vee \rho^{d-\frac{2d}{p}}y^{\frac2p-1}>1
 \,\Leftrightarrow\,
 \rho^{d-\frac{2d}{p}}y^{\frac2p}\ge y\,.
 $$
 With $y=\la(\del^{\overline{0\rho}}\UT)$ (note $y\ge\rho^d$ if $y\ne0$) it follows that
 \begin{equation} \label{eq:lagelap}
  \la(\del^{\overline{0\rho}}\UT) \lle \rho^{d-\frac{2d}{p}}\la\big(\del^{\overline{0\rho}}\UT\big)^{\frac{2}{p}}\,.
 \end{equation}
 Inserting the combination of (\ref{eq:distvsod}) and (\ref{eq:lagelap}) as well as (\ref{eq:dv}) 
 into the squared version of (\ref{eq:V-Rledist}) yields
 \begin{equation*} 
  \|V-R\|^2_{L^2(\I)} \lle
  2\Ceins^2 \|\dist(V,\sod)\|^2_{L^2(\UT)} + 2N^{d-\frac{2d}{p}+2} \cxxxvii(p) \la\big(\del^{\overline{0\rho}}\UT\big)^{\frac{2}{p}}
 \end{equation*} 
for some constant $\cxxxvii(p)>0$. Thus the lemma follows by a little rearrangement and renaming of constants.
\end{proof}

\begin{lem} \label{lem:s0-estimate}
 For all $m\ge m_0$, $\s>0$, $N\in\N$ and $\om\in\Om$ there exists a random rotation $R=R(\om)\in\sod$ such that
 \begin{eqnarray*}
  H_\sn(\om) - H_\sn(\sta) &\ge&
  \cxix\, \|V-R\|^2_{L^2(\I)}+ \ciii\big(\s-\s_0(N,m)\big)\big(|\P|-\sumigt\big)
 \end{eqnarray*}
 with $ \s_0(N,m) = \cxxvii N^2 + \cxxviii m + \cxxix$ for some constants $\cxxvii,\cxxviii>0$ and $\cxxix\in\R$.
\end{lem}
\begin{proof}
 Lemma~\ref{lem:dist-estimate} and Lemma~\ref{lem:la-estimate} together state that
 $$ H(\om)-H(\sta)  \gge
   \cxix \|V-R\|^2_{L^2(\I)} + \big(\s\ciii-\cxvii m- \cxviii\big)\!\big(|\P|-\sumigt\big) 
  -\cxx(p)N^{2+d-\frac{2d}{p}}\la\big(\del^{\overline{0\rho}}\UT\big)^{\frac2p}$$
 for all $p\in[2d/(2+d),2]$. Therefore we have to estimate $\la(\del^{\overline{0\rho}}\UT)^{\frac2p}$ from above. We start the estimate with Lemma~\ref{lem:estBoundary}(\ref{asc}) to get a bound in terms of $|\del\T|$. Then we use two different bounds: On the one hand we use Lemma~\ref{lem:estBoundary}(\ref{asb}), i.e.\ $|\del\T|\le\cix(|\P|-\sum_{i\in I}\g_i|\T^i|)$, and on the other hand we use the bound $|\del\T|\le|\T|\le\cxiv N^d$, provided by Lemma~\ref{lem:estTPBoxNd}(\ref{ase}). This yields
 $$ \la\big(\del^{\overline{0\rho}}\UT\big)^{\frac2p}
    \lle \big(\cx |\del\T|\big)^{1+(\frac2p-1)}
    \lle  \cx^{\frac2p}\, \cix \big(|\P|-\sumigt\big)\big(\cxiv N^d\big)^{\frac2p-1}
    $$ 
 Since $N^{2+d-\frac{2d}{p}}N^{\frac{2d}{p}-d}=N^2$ for all $p$, the choice of $p$ does not matter. We choose $p=2$ (i.e.\ $\frac{2}{p}=1$). Setting $\cxxvii:=\cxx(2)\cx\cix/\ciii$, we conclude
 $$ H(\om)-H(\sta)  \gge  \cxix \|V-R\|^2_{L^2(\I)} + \big(\s\ciii-\cxvii m- \cxviii-\cxxvii\ciii N^2\big)\big(|\P|-\sumigt\big)\,,  $$
 which implies the lemma with $\cxxviii:=\cxvii/\ciii>0$ and $\cxxix:=\cxviii/\ciii\in\R$.
\end{proof}

Let us remark that a choice $p>2$ would give a worse result. Though Lemma~\ref{lem:la-estimate} would also work with an additional $\la(\del^{\overline{0\rho}}\UT)^1$-term, the factor $N^{2+d-\frac{2d}{p}}$ would be worse than $N^2$ and could not be compensated by $|\del\T|^{\frac2p-1}$ since $|\del\T|$ may be small.

\begin{proof}[Proof of Lemma \ref{lem:H-estimate}]
 This lemma is an immediate corollary to Lemma~\ref{lem:s0-estimate} since $\s-\s_0(N,m)\ge0$ and $|\P|-\sum_{i\in I}\g_i|\T^i|\ge0$ by Lemma~\ref{lem:estBoundary}(\ref{asa}). 
\end{proof}

\subsubsection{A Lower Bound for the Partition Sum} \label{sec:partitionsum}

In this section we prove the following lower bound of the partition sum, which is an analogue to \cite[Lemma~3.1]{hmr13}.
\begin{lem} \label{lem:partitionsum}
 For all $\g>0$ and $m\in\R$ there exist a constant $\cxxii(\g,m)>0$ and an $N_0(\g,m)\in\N$ such that for all $N\ge N_0(\g,m)$, $\be>0$ and $\s>0$ one has
 $$ Z_\bsn \gge e^{-N^d[\be\g+\cxxii(\g,m)]}e^{-\be H_\sn(\sta)}\,.$$
\end{lem}
\begin{proof}
 The proof uses the idea of the proof of \cite[Lemma~3.1]{hmr13}, namely to restrict the integral to a set of blurred configurations. But we have to blur a configuration slightly differently to the standard configuration since we have to ensure that the Lebesgue measure of the blurred tiles is not smaller than the Lebesgue measure of the corresponding standard tile.

 We start the proof with some preliminaries. Let us recall that $\M$ is $B_0$-periodic for some box $B_0$, which is the image of the cube $[0,1]^d$ under some linear map $L$. For $r\in(0,\tfrac{\ep}{4})$ and $N\in\N$ such that $\lfloor\tfrac{N}{1+r}\rfloor \ge \lceil\tfrac{4}{\ep}\rceil=:n_0$ we define a configuration $\sta_r$ with vertices $\Q_r$ and tiles $\U_r$ as follows (we suppress the $N$-dependency in the notation). It looks almost like the given tessellation $\M$, but is a bit enlarged. The domain $\I$ is partitioned into boxes $B_\k$, $\k\in\{1,\ldots,\lfloor\tfrac{N}{1+r}\rfloor\}^d$ which are slight enlargements of $B_0=L\big[[0,1]^d\big]$. In each box-direction $Le_j$ (with unit vector $e_j$), there are $\lfloor \tfrac{N}{1+r} \rfloor-n_0$ boxes scaled by the factor $(1+r)$, followed by $n_0$ boxes scaled by $O/n_0$, with ``off-cut'' $O:=N-(1+r)(\lfloor\tfrac{N}{1+r}\rfloor-n_0)$. Thus box $B_\k$, with $\k\in\{1,\ldots,\lfloor\tfrac{N}{1+r}\rfloor\}^d$, has length $\l^\k_j$ in box direction $Le_j$, $j\in\{1,\ldots,d\}$, where  $\l^\k_j=|Le_j|(1+r)$ if $1\le\k_j\le\lfloor \tfrac{N}{1+r} \rfloor-n_0$ and $\l^\k_j=|Le_j|\,O/n_0$ else. Now $\sta_r$ is defined such that $\Q_r\!\!\upharpoonright_{B_\k}=\V(\l^\k\cdot\M)$ and similarly for $\U_r$. At the separation hyperplanes between the scales, the points are moved a little bit, such that all tiles, which intersects such a separation hyperplane are scaled like the box which is to the ``left'' in the corresponding coordinate direction.
 
 Figure~\ref{fig:star}
 \begin{figure}
  \begin{center}
   \includegraphics[width=0.9\textwidth]{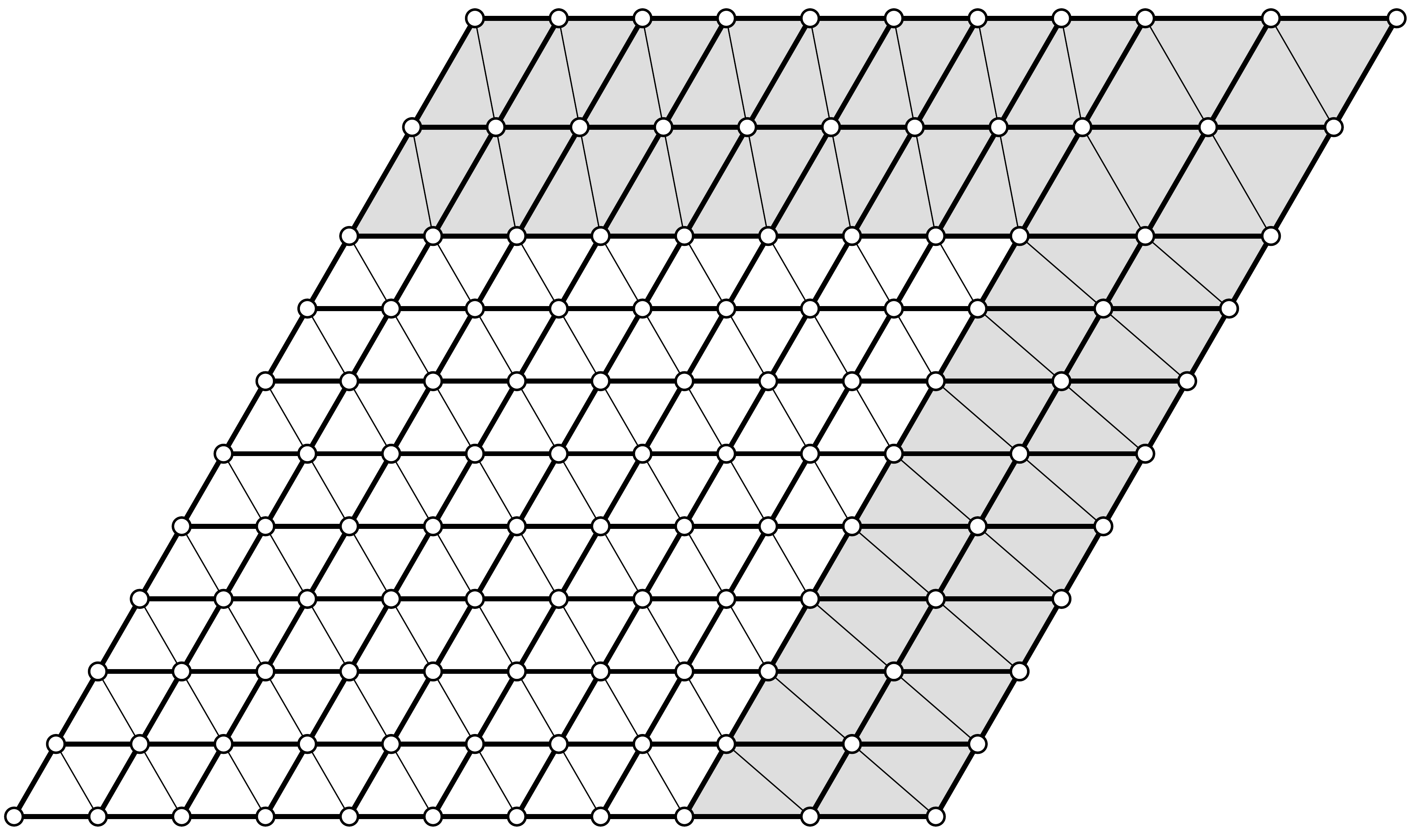}
  \end{center}
  \caption{The configuration $\sta_r$ for the triangular lattice \label{fig:star}}
 \end{figure}
 illustrates the configuration $\sta_r$ in the case where $\M$ is the triangular lattice. In that case, $B_0$ is a rhombus consisting of two triangles. The white boxes are the boxes scaled by $1+r$ and build the ``bulk''. In contrast, the grey shaded boxes are in the ``off-cut'' and scaled by larger factors which may also differ in different directions. We have to use this ``off-cut-boxes'' to ensure that $\T(\sta_r)$ completely fills the domain $\I$, whose size is a natural number times the size of $B_0$ (in each direction). 
 
 Moreover, we blur the configuration $\sta_r$ a little bit and define the set
 $$ A_r \,:=\, \big\{\om\in\ti{\Om}\mid \exi\text{bijective } f:\P(\om)\to\Q_r: \fa x\in\P(\om):|x-f(x)|<\tfrac{r}{2}\big\} $$
 of all configurations whose points are $r/2$-close to $\Q_r$. Then we claim that all configurations in $A_r$ are admitted configurations without any defect, i.e.\ $A_r\subset\Om$ and $\del\P=\emptyset$ on $A_r$. Since
 $$
 O = N-(1+r)(\lfloor\tfrac{N}{1+r}\rfloor-n_0)\,\begin{dcases}
                                              \lle N-(1+r)(\tfrac{N}{1+r}-1-n_0) = (1+r)(1+n_0) \\
                                              \gge N-(1+r)(\tfrac{N}{1+r}-n_0) = (1+r)n_0
                                             \end{dcases} $$
 we conclude $1+r\le \frac{O}{n_0}$ and
 $$ 1 \lle 1+r-2\cdot\tfrac{r}{2} \lle \tfrac{O}{n_0}+2\cdot\tfrac{r}{2}  \lle (1+r)(1+\tfrac{1}{n_0}) + r = 1+2r+\tfrac{1}{n_0}+\tfrac{r}{n_0} \lle 1+\ep$$
 as $n_0\ge\frac4\ep$ and $r\le\frac\ep4$.  By the definition of the set $A_r$, the distance between two points in $\P(\om)$ for any $\om\in A_r$ is in $[1+r-2\cdot\frac r2,\frac{O}{n_0}+2\cdot\frac r2]$ times the distance of the corresponding points in $\Q$. Thus the estimate above shows that all tiles are, up to translation, in $\Uept$ for some $i\in I$ and the claim follows.
 
 Furthermore, there exists a constant $\cxxxviii>0$ such that
 \begin{equation} \label{eq:partsum1}
  \big|\hloc^i(\t)-\hloc^i(\st^i)\big| \lle \cxxxviii 
 \end{equation}
  for all $i\in I$ and $\t\in\Uept$ since the image of the compact set $\Uept$ is compact as $\hloc^i$ is continuous.
 
 Now we begin with the actual proof. Let $\g>0$ and $m\in\R$. We choose $r\in(0,\frac\ep4)\cap(0,\frac{1}{2})$ so small that
 \begin{equation} \label{eq:partsum2}
  r(md\cxliv +\cxlvi) \lle \tfrac\g3 
 \end{equation}  
 for some constants $\cxliv,\cxlvi>0$ defined below and that for all $i\in I$ and $\t\in\Urt$
 \begin{equation} \label{eq:partsum3}
  \big|\hloc^i(\t)-\hloc^i(\st^i)\big| \lle \tfrac{\g}{3\cxiv}\,,
 \end{equation}  
 which is possible since $\hloc^i$, $i\in I$, are continuous. Furthermore, we choose $N_0\in\N\setminus\{1\}$ large enough such that $\lfloor\tfrac{N_0}{1+r}\rfloor \ge \lceil\tfrac{4}{\ep}\rceil=n_0$ and such that 
 \begin{equation} \label{eq:partsum4}
 \tfrac{1}{N_0}\big(\cxxxviii\cxliii+md\cxliv+\cxlvi\big) \lle \tfrac\g3 
 \end{equation}
 for some constant $\cxliii>0$ defined below.
 Let $N\ge N_0$. Now we estimate $\mu(A_r)$, where $A_r$ is the set of blurred configurations defined above. Since the number of points is Poisson distributed and independent from the location of the points, which are iid and uniformly distributed, it follows that
 \begin{eqnarray}
  \mu(A_r) &=&
  \mu(|\P|=|\Q_r|) \cdot \frac{\la(U_\frac{r}{2}(0))}{\la(\I)}\,|\Q_r| \cdot \frac{\la(U_\frac{r}{2}(0))}{\la(\I)}\,(|\Q_r|-1) \cdot \ldots \cdot \frac{\la(U_\frac{r}{2}(0))}{\la(\I)}\,1 \notag\\
  &=& e^{-\la(\I)}\frac{\la(\I)^{|\Q_r|}}{|\Q_r|!}\cdot \Big(\frac{\la(U_\frac{r}{2}(0))}{\la(\I)}\Big)^{|\Q_r|}\cdot|\Q_r|!
  \,=\,
  e^{-\cxiii N^d+|\Q_r|\log\la(U_\frac{r}{2}(0))} \notag\\
  &\ge&
  e^{-N^d\cdot\cxlii(r)}
  \label{eq:muAr}
 \end{eqnarray}
 for some constant $\cxlii(r)>0$ only depending on $r$ since $\la(\I)=\cxiii N^d$ and $|\Q_r|\le\cxc N^d$ by Lemma~\ref{lem:estTPBoxNd}, Assertions (\ref{asg}) and (\ref{asd}). Note that   $\cxlii(r)\to\infty$ since $\la(U_\frac{r}{2}(0))\to0$ as $r\to0$.
 
 In the following, we estimate the difference of the Hamiltonians of any configuration in $A_r$ and the standard configuration. Thereto we call $\T_\text{bulk}$ the set of tiles which are in a box which is scaled by $(1+r)$ in all directions. The set of all other tiles is called $\T_\text{off}$. Since $n_0$ is fixed, there is a uniform constant $\cxliii>0$ such that $|\T_\text{off}|\le\cxliii N^{d-1}$. It follows that, for all $\om\in\A_r$,
 \begin{eqnarray}
  H(\om)-H(\sta) 
  &=&
  \sum_{\t\in\T_\text{bulk}}\big(\hloc^{\typts}(\t)-\hloc^{\typts}(\st^{\typts})\big)+\sum_{\t\in\T_\text{off}}\big(\hloc^{\typts}(\t)-\hloc^{\typts}(\st^{\typts})\big) \notag \\
  &&+\sum_{i\in I}|\U_r^i|\hloc^i(\st^i) - \sum_{i\in I} |\U^i|\hloc^i(\st^i) + \s0-\s0-m|\Q_r|+m|\Q| \notag\\
  &\le&
  \tfrac{\g}{3\cxiv}|\U_r|+\cxxxviii\cxliii N^{d-1} - \sum_{i\in I}\big(|\U^i|-|\U_r^i|\big)\hloc^i(\st^i) + m(|\Q|-|\Q_r|) \label{eq:HAr-Hsta1}
 \end{eqnarray}
 using also the estimates (\ref{eq:partsum3}), $|\T_\text{bulk}|\le|\U_r|$ and (\ref{eq:partsum1}). Let $\cxliv:=|\V(\M\!\!\upharpoonright_{B_0})|$ and $\cxlv:=|\U^i\cap\M\!\!\upharpoonright_{B_0}\!\!|$ be the number of vertices and tiles of type $i$, respectively, in $B_0$ (of the standard configuration $\sta$). Then:
 \begin{align*}
  |\Q|&=\cxliv N^d\,, & |\Q_r| &= \cxliv \big\lfloor \tfrac{N}{1+r}\big\rfloor^d\,, &
  |\U^i|&=\cxlv N^d &&\text{and} & |\Q_r^i| &= \cxlv \big\lfloor\tfrac{N}{1+r}\big\rfloor^d\,.
 \end{align*}
 Therefore we estimate using $\frac{1}{1+r}\ge(1-r)$ and $(1-x)^d\ge1-dx$ for $x=r+\frac1N\in(0,1)$, which can be derived with Taylor expansions,
 \begin{eqnarray}
  |\Q|-|\Q_r|
  &=&
  \cxliv\big( N^d - \lfloor\tfrac{N}{1+r}\rfloor^d\big)
  \lle
  \cxliv N^d\big(1 - \big(\tfrac{1}{1+r}-\tfrac1N\big)^d\big) \notag\\
  &\le&
  \cxliv N^d\big(1 - \big(1-r-\tfrac1N\big)^d\big)
  \lle
  \cxliv N^d\big(1 - \big(1-d\big(r+\tfrac1N\big) \big)\big) \notag\\
  &=&
  d\cxliv \big(r+\tfrac1N\big) N^d\,. \label{eq:taylor} 
 \end{eqnarray}
 Analogously, we receive 
 $$ |\U^i|-|\U_r^i|\lle d\cxlv \big(r+\tfrac1N\big)N^d \,,$$
 which implies
 \begin{eqnarray}
  - \sum_{i\in I}\big(|\U^i|-|\U_r^i|\big)\hloc^i(\st^i)
  &\le&
  \sum_{i\in I}\big(|\U^i|-|\U_r^i|\big)\big|\hloc^i(\st^i)\big| \notag\\
  &\le&
  \sum_{i\in I}d\cxlv \big(r+\tfrac1N\big)N^d\big|\hloc^i(\st^i)\big|
  \,=\,
  \cxlvi \big(r+\tfrac1N\big) N^d \label{eq:Ui-Uir}
 \end{eqnarray}
 with $\cxlvi:=\sum_{i\in I}d\cxlv|\hloc^i(\st^i)|>0$. Using $|\U_r|\le\cxiv N^d$, \eqref{eq:Ui-Uir} and \eqref{eq:taylor} for the first inequality and \eqref{eq:partsum4} and \eqref{eq:partsum2} for the second inequality, we continue the estimate (\ref{eq:HAr-Hsta1}) as follows:
 \begin{eqnarray}
  H(\om)-H(\sta) 
  &\le&
  \tfrac{\g}{3\cxiv}\cxiv N^d+\cxxxviii\cxliii N^{d-1} + \cxlvi \big(r+\tfrac1N\big) N^d + md\cxliv \big(r+\tfrac1N\big) N^d \notag\\
  &=&
  N^d\big(\tfrac\g3 + \tfrac1N\big(\cxxxviii\cxliii+md\cxliv+\cxlvi\big) + r (md\cxliv+\cxlvi)\big) \notag\\
  &\le&
  N^d\big(\tfrac\g3 + \tfrac\g3 + \tfrac\g3 \big) \,=\, \g N^d \,.
  \label{eq:HAr-Hsta2}
 \end{eqnarray} 
 Finally, we estimate the partition sum using first (\ref{eq:HAr-Hsta2}) and then (\ref{eq:muAr}) to conclude the proof:
 \begin{eqnarray*}
   Z_\bsn
   &=&
   \int_\Om e^{-\be H(\om)}\,\mu(d\om)
   \gge
   e^{-\be H(\sta)} \int_{A_r} e^{-\be (H(\om)-H(\sta))}\,\mu(d\om) \\
   &\ge&
   e^{-\be H(\sta)} e^{-\be \g N^d} \mu(A_r) \\
   &\ge&
   e^{-\be H(\sta)} e^{-\be \g N^d}  e^{-N^d\cdot\cxlii(r)} 
   \,=\, 
   e^{-N^d[\be\g+\cxxii(\g,m)]}e^{-\be H(\sta)}
 \end{eqnarray*}
 with $\cxxii(\g,m):=\cxlii(r(\g,m))>0$. Note that $\cxxii(\g,m)\to\infty$ as $\g\to0$ or $m\to\infty$.
\end{proof}

\subsubsection{An Upper Bound for the Internal Energy} \label{sec:internalenergy}

In this section, we obtain an estimate of $E_\bsn[\frac{1}{|\T|}(H_\sn(\cdot)-H_\sn(\sta)]$. Thereto we will need a lower bound on $\|V-R\|^2_{L^2(\I)}$, which can be expressed as a sum over all edges of all tiles in $\T$. It turns out that taking the sum just over all edges of a suitable spanning tree is enough. But we also need to trace back the types of the edges. Therefore we introduce spanning trees of $\T$ labelled by edge types.

We define the label set $\Sigma$ as the union of all edges of $\st^i$, $i\in I$, regarded as vectors in $\R^d$ (each edge induces two vectors with opposite orientation). Let $\trees{n}$ be the set of trees with $n$ vertices labelled by elements of $\Sigma$. We denote the label of a vertex $k$ by $\xi_k$. We can consider a tree $T\in\trees{n}$ as a rooted tree with root $1$. Then, for each $l\in\{2,\ldots,n\}$, there exists a unique $k_T(l)\in\{1,\ldots,l\!-\!1\}$ such that $k_T(l)\sim l$ in $T$. For $\om\in\Om$, we define the function $\no: \{1,\ldots,|\P\setminus\Pe|\}\to\{k\in\N\mid X_k\in\P\setminus\Pe\}$ as the unique increasing bijection between these sets. 

For a labelled tree $T\in\trees{n}$ and $\om\in\Om$ with $n=|\P\setminus\Pe|$, we define the graph $G(T,\T)$ as follows: The vertex set is $\{X_{\no(k)}, k=1,\ldots,n\}$; two such vertices $X_{\no(k)}$ and $X_{\no(l)}$, $1\le k<l\le n$, form an edge, if $k=k_T(l)$ (i.e.\ $k\sim l$ in $T$) and if there is a tile $\t\in\T$ such that $X_{\no(k)},X_{\no(l)}\in\t$ and $\xi_l=v_{\scriptscriptstyle\t}(X_{\no(l)}) - v_{\scriptscriptstyle\t}(X_{\no(k)})$, where $v_{\scriptscriptstyle\t}:\t\to\st^{\typts}$ is the affine linear map defined in (\ref{eq:defvt}). Thus $G(T,\T)$ can be viewed as a graph isomorphic to a sub-graph of $T$ using vertices of $\P\setminus\Pe$ such that the label of a vertex coincide with the type of an adjacent edge in $\T$.  

A labelled tree $T\in\trees{|\P\setminus\Pe|}$ is called a \emph{labelled spanning tree} of $\T$ if $G(T,\T)$ is a spanning tree of $\T$, viewed as a graph with vertices $\P\setminus\Pe$ and edges formed by the edges of all tiles. In that case we write $T\bowtie\T$.
Since $\UT$ is connected, there exists a labelled spanning tree $T\in\trees{|\P\setminus\Pe|}$: just take any spanning tree and label the vertices accordingly level by level, beginning with the vertices adjacent to the root (whose label is irrelevant).

\begin{lem} \label{lem:spanningtree}
 There is a constant $\cxxiii>0$ such that for all $N\in\N$, $R\in\sod$, $\om\in\Om$ and  $T\in\trees{|\P\setminus\Pe|}$ with $T\bowtie\T$ the following estimate holds:
 $$ \|V-R\|^2_{L^2(\I)} \gge \cxxiii \sum_{l=2}^{|\P\setminus\Pe|} \big|(X_{\no(l)}-X_{\no(k_T(l))})-R^t\xi_l\big|^2 $$
\end{lem}
\begin{proof}
 Let $\t\in\T$. Let $\Sim(\t):=\{v_{\scriptscriptstyle\t}^{-1}[\st^{\typts,j}]\mid j=1,\ldots,J_\typts\}$ be the set of simplices on which $v_{\scriptscriptstyle\t}$ is affine linear. For a simplex $\triangle\in\Sim(\t)$, let $\substack{{\scriptscriptstyle\triangle}\\[-0.2ex]{\textstyle\e}}:\{0,\ldots,d\}\to\{k\in\N\mid X_k \text{ is a vertex of }\triangle\}$ be the unique increasing bijection between these sets.
 
 In the following estimate for a single simplex we simple write $x_k$ for $X_{\substack{{\scriptscriptstyle\triangle}\\[-0.3ex]\e}(k)}$, $k=0,\ldots,d$. We have
 $$ v_{\scriptscriptstyle\t}(x) \,=\, V_{\!\scriptscriptstyle\triangle} x + z_{\scriptscriptstyle\triangle}\,,\qquad x\in\triangle \,,$$
 for some $V_{\!\scriptscriptstyle\triangle}\in\R^\dxd$ and $z_{\scriptscriptstyle\triangle}\in\R^d$ since $v_{\scriptscriptstyle\t}$ is affine linear on $\triangle$. Using this and $|Ry|=|y|$ because of $R\in\sod$, it follows that
 \begin{align}
  &\sum_{0\le k<l\le d} \big|(x_l-x_k)-R^t(v_{\scriptscriptstyle\t}(x_l)\!-\!v_{\scriptscriptstyle\t}(x_k))\big|^2 
  \,\,=\,
  \sum_{0\le k<l\le d} \big|R(x_l-x_k)-(v_{\scriptscriptstyle\t}(x_l)\!-\!v_{\scriptscriptstyle\t}(x_k))\big|^2  \notag\\
  &\qquad\quad=\,
  \sum_{0\le k<l\le d} \big|R(x_l-x_k)-((V_{\!\scriptscriptstyle\triangle} x_l+z_{\scriptscriptstyle\triangle})-(V_{\!\scriptscriptstyle\triangle} x_k+z_{\scriptscriptstyle\triangle}))\big|^2  \notag\\
  &\qquad\quad\le\,
  \sum_{0\le k<l\le d} |R-V_{\!\scriptscriptstyle\triangle}|^2 |x_l-x_k|^2
  \lle
  \cxlvii |R-V_{\!\scriptscriptstyle\triangle}|^2
  \label{eq:SpanSingleSim}
 \end{align}
 for some uniform constant $\cxlvii>0$ since the size of a tile is uniformly bounded.
 
 Therefore we can estimate using the fact that the size of a simplex is uniformly bounded
 \begin{eqnarray*}
  \|V-R\|^2_{L^2(\I)}
  &\ge&
  \sum_{\t\in\T} \|V-R\|^2_{L^2(\t)}
  \,=\,
  \adjustlimits\sum_{\t\in\T} \sum_{\triangle\in\Sim(\t)} \la(\triangle)|V_{\!\scriptscriptstyle\triangle} - R|^2 \\
  &\!\!\!\!\stackrel{(\ref{eq:SpanSingleSim})}{\ge}\!\!\!\!&
  \adjustlimits\sum_{\t\in\T} \sum_{\triangle\in\Sim(\t)} \frac{\la(\triangle)}{\cxlvii}
   \!\!\sum_{0\le k<l\le d}\!\! \big|(X_{\substack{{\scriptscriptstyle\triangle}\\[-0.3ex]\e}(l)}\!-\!X_{\substack{{\scriptscriptstyle\triangle}\\[-0.3ex]\e}(k)})-R^t(v_{\scriptscriptstyle\t}(X_{\substack{{\scriptscriptstyle\triangle}\\[-0.3ex]\e}(l)})\!-\!v_{\scriptscriptstyle\t}(X_{\substack{{\scriptscriptstyle\triangle}\\[-0.3ex]\e}(k)}))\big|^2\\
   &\ge&
   \cxxiii \sum_{l=2}^n \big|(X_{\no(l)}-X_{\no(k_T(l))})-R^t\xi_{l}\big|^2
 \end{eqnarray*}
 for some $\cxxiii>0$. We obtained the last inequality by restricting the sum, which is taken over all edges of all simplices, to edges in $G(T,\T)$; note that $\xi_{\no^{-1}(\substack{{\scriptscriptstyle\triangle}\\[-0.3ex]\e}(l))}=v_{\scriptscriptstyle\t}(X_{\substack{{\scriptscriptstyle\triangle}\\[-0.3ex]\e}(l)})-v_{\scriptscriptstyle\t}(X_{\substack{{\scriptscriptstyle\triangle}\\[-0.3ex]\e}(k)})$ if $\{X_{\substack{{\scriptscriptstyle\triangle}\\[-0.3ex]\e}(k)},X_{\substack{{\scriptscriptstyle\triangle}\\[-0.3ex]\e}(l)}\}$ is an edge of $G(T,\T)$ since $T\bowtie\T$.
\end{proof}

The following lemma is an analogue to \cite[Lemma~3.3]{hmr13}.
\begin{lem} \label{lem:EH-H}
 There exist constants $\cxxiv>0$ and $\be_0>0$ such that for all $m\ge m_0$ and all $\de>0$ there exist $N_0(\de,m)\in\N$ and $\cxxv(\de,m)\in\R$ such that for all $N\ge N_0$, $\s\ge\s_0(N,m)$ and $\be\ge\be_0$ the following estimate holds:
 $$ E_\bsn\big[ \tfrac{1}{|\T|}\big(H_\sn(\cdot)-H_\sn(\sta)\big)\big]
 \lle \de + \tfrac1\be \exp\big(-N^d\big[\be\tfrac{c_0}{8}\de + \cxxiv\log\be - \cxxv(\de,m)\big]\big) $$
\end{lem}

\begin{proof}
 We use some ideas of the proof of \cite[Lemma~3.3]{hmr13}.
 Let $\de>0$ and $m\ge m_0$. We set $N_0(\de,m)=N_0(\g,m)$ as in Lemma~\ref{lem:partitionsum} with $\g=\tfrac{c_0}{8}\de$. Let $N\ge N_0(\de,m)$ and $\s\ge\s_0(N,m)$. We set
 \begin{eqnarray*}
  \Om^{>\de} &:=& \big\{\om\in\Om: H_\sn(\om)-H_\sn(\sta)>\de|\T|\big\} \qquad\text{and}\\
  \Om^{\le\de} &:=& \big\{\om\in\Om: H_\sn(\om)-H_\sn(\sta)\le\de|\T|\big\} \,. 
 \end{eqnarray*}
 First we estimate
 \begin{equation} \label{eq:Qless}
  E_\bsn\big[ \tfrac{1}{|\T|}\big(H(\cdot)-H(\sta)\big)\1_{\Om^{\le\de}}\big]
  \lle
  E_\bsn\big[ \tfrac{1}{|\T|}\de|\T|\1_{\Om^{\le\de}}\big]
  \lle
  \de\,.
 \end{equation}
 
 The estimate on $\Om^{>\de}$ is much more involved. Using the inequality $xe^{-x}\le e^{-x/2}$ for $x=\be(H(\om)-H(\sta))$ and $|\T|\ge1$, we estimate similarly as in the proof of Markov's Inequality (writing shortly $E$ for $E_\bsn$):
 \begin{eqnarray}
  E\big[ \tfrac{1}{|\T|}\big(H(\cdot)-H(\sta)\big)\1_{\Om^{>\de}}\big]
  &=&
  \frac{e^{-\be H(\sta)}}{Z_\bsn} \int_{\Om^{>\de}} \tfrac{1}{|\T|}\big(H(\om)-H(\sta)\big) e^{-\be(H(\om)-H(\sta))} \,\mu(d\om) \notag\\
  &\le&
  \frac{e^{-\be H(\sta)}}{\be Z_\bsn} \int_{\Om^{>\de}} e^{-\frac\be2(H(\om)-H(\sta))} e^{\frac\be4(H(\om)-H(\sta)-\de|\T|)}\,\mu(d\om) \notag\\
  &\le&
  \frac{e^{-\be H(\sta)}}{\be Z_\bsn} \int_{\Om} e^{-\frac\be4(H(\om)-H(\sta)+\de|\T|)}\,\mu(d\om)\notag\\
  &\le&
  \frac{e^{-\be H(\sta)}}{\be Z_\bsn}\, e^{-\frac\be4\de\co N^d} \int_{\Om} e^{-\frac\be4 \cxix\|V-R\|^2_{L^2(\I)}}\,d\mu\,, \label{eq:EH-H1}
 \end{eqnarray}
 where we used Lemma~\ref{lem:H-estimate} and $|\T|\ge\co N^d$ in the last step. Now we partition $\Om$ into
 $ \Om_n := \{\om\in\Om: |\P\setminus\Pe|=n\}$, $n\in\N$.  Using Lemma~\ref{lem:spanningtree}, we estimate the integral in the last line restricted to $\Om_n$
 \begin{eqnarray}
   \lefteqn{\int_{\Om_n} e^{-\frac\be4 \cxix\|V-R\|^2_{L^2(\I)}}\,d\mu}\quad \notag\\
   &\le&
   \sum_{T\in\trees{n}} \int_{\Om_n} \1_{T\bowtie\T}\exp\!\Big[\!-\!\tfrac\be4 \cxix\cxxiii \sum_{l=2}^n \big|(X_{\no(l)}-X_{\no(k_T(l))})-R^t\xi_l\big|^2\Big]d\mu\ \notag\\
   &\le&
   \sum_{T\in\trees{n}} \int_{\I^n} \exp\!\Big[\!-\!\tfrac\be4 \cxix\cxxiii  \sum_{l=2}^n \big|(x_l-x_{k_T(l)})-R^t\xi_l\big|^2\Big] \frac{dx_1}{\la(\I)}\cdots\frac{dx_n}{\la(\I)}
   \label{eq:intV-R1}
 \end{eqnarray}
 where we used $\1_{T\bowtie\T}\le1$ and the fact that $X_{\no(k)}$, $1\le k\le n$, are independent and uniformly distributed on $\I$.
 For each tree $T$, we define the matrix $M_T=(M_{kl})_{kl}\in\R^\nxn$ as follows: $M_{kk}=1$, $M_{kl}=-1$ if $k=k_T(l)$ and $M_{kl}=0$ else. Then $\det M_T=1$ since all diagonal entries are $1$ and $M_T$ is a lower triangular matrix as $k_T(l)<l$. Using the transformation
 $$ y = M_T x - R^t\xi $$
 with $x=(x_1,\ldots,x_n)^t$, $y=(y_1,\ldots,y_n)^t$ and $\xi=(0,\xi_2,\ldots,\xi_n)^t$,  we continue 
 \begin{eqnarray}
  (\ref{eq:intV-R1})
  &=&
  \sum_{T\in\trees{n}} \frac{1}{\la(\I)^n} \int_{y[\I^n]} \exp\!\Big[\!-\!\tfrac\be4 \cxix\cxxiii  \sum_{l=2}^n |y_l|^2\Big] \,dy_1\ldots dy_n \notag\\
  &\le&
  \sum_{T\in\trees{n}} \frac{\la(y_1[\I])}{\la(\I)^n} \Big(\int_{\R^d} e^{-\frac\be4 \cxix\cxxiii|y_2|^2} \,dy_2\Big)^{n-1} \notag\\
  &=&
  \sum_{T\in\trees{n}} \frac{1}{\la(\I)^{n-1}} \Big(\frac{1}{\be\cxlviii}\Big)^{\frac{d}{2}(n-1)} 
  =
  n^{n-2}|\Sigma|^n \big(\la(\I)(\be\cxlviii)^{\frac{d}{2}}\big)^{-(n-1)}
  \label{eq:intV-R2}
 \end{eqnarray}
 with $\cxlviii:= \cxix\cxxiii/(8\pi)$. In the last line we used first $y_1[\I]=\I$ and second $|\trees{n}| = n^{n-2}|\Sigma|^n$ by Cayley's formula. 
 
 Lemma~\ref{lem:estTPBoxNd}, Assertions~(\ref{asg}) and (\ref{asd}), state that $\la(\I)=\cxiii N^d$ and $\Om_n=\{|\P\setminus\Pe|=n\}=\emptyset$ if $n\notin A:= [\cxvi N^d,\cxc N^d]\cap\N$, respectively. Therefore (\ref{eq:intV-R2}) implies, with $\cxlix=\cxlviii\!^{\frac d2}\cxiii/(\cxc\Sigma)$,
 \begin{eqnarray}
  \int_{\Om} e^{-\frac\be4 \cxix\|V-R\|^2_{L^2(\I)}}\,d\mu\
  &\le&
  \sum_{n\in A}  n^{n-2}|\Sigma|^n \big(\la(\I)(\be\cxlviii)^{\frac{d}{2}}\big)^{-(n-1)} \notag\\
  &\le&
  \sum_{n\in A} |\Sigma|^{-1}\Big(\frac{\cxc N^d |\Sigma|}{\cxiii N^d(\be\cxlviii)^{\frac{d}{2}}}\Big)^{n-1}
  \,=\,
  \sum_{n\in A} |\Sigma|^{-1} \big(\cxlix\be^{\frac{d}{2}}\big)^{-(n-1)} \notag\\
  &\le&
  (\cxc-\cxvi)N^d|\Sigma|^{-1}\big(\cxlix\be^{\frac{d}{2}}\big)^{-(\cxvi N^d-1)} \notag\\
  &\le&
  e^{-N^d[\cxxiv\log\be-\cl]}
  \label{eq:intV-R3}
 \end{eqnarray}
 for $\be\ge\be_0:=\cxlix\!^{-\frac2d}>0$ and some constants $\cxxiv>0$ and $\cl\in\R$.
 
 Using Lemma~\ref{lem:partitionsum} and (\ref{eq:intV-R3}), we estimate (\ref{eq:EH-H1}) further:
 \begin{eqnarray}
  E_\bsn\big[ \tfrac{1}{|\T|}\big(H(\cdot)-H(\sta)\big)\1_{\Om^{>\de}}\big]
  &\le&
  \tfrac1\be e^{+N^d[\be\g+\cxxii(\g,m)]} e^{-\frac\be4\de\co N^d} e^{-N^d[\cxxiv\log\be-\cl]} \notag\\
  &=&
  \tfrac1\be \exp\big(-N^d\big[\be\tfrac\co4\de-\be\g+\cxxiv\log\be-\cl-\cxxii(\g,m)\big]\big) \notag\\
  &=&
  \tfrac1\be \exp\big(-N^d\big[\be\tfrac\co8\de+\cxxiv\log\be-\cxxv(\de,m)\big]\big)
  \label{eq:Qge}
 \end{eqnarray}
 with $\g=\frac{c_0}{8}\de$ and $\cxxv(\de,m)=\cl+\cxxii(\frac{c_0}{8}\de,m)\in\R$. The combination of (\ref{eq:Qless}) and (\ref{eq:Qge}) yields the conclusion of the lemma. 
\end{proof}

\subsubsection{Results} \label{sec:results}

\begin{corol} \label{cor:limits}
 The following statements hold for all $m\ge m_0$:
 \begin{alignat}{2}
  \adjustlimits\lim_{\be\to\infty} \limsup_{N\to\infty} \sup_{\s\ge\s_0(N,m)} &
  E_\bsn\big[ \tfrac{1}{|\T|}\big(H_\sn(\cdot)-H_\sn(\sta)\big)\big] &&\,=\, 0 \label{eq:lim1}\\
  \adjustlimits\lim_{\be\to\infty} \limsup_{N\to\infty} \sup_{\s\ge\s_0(N,m)} &
  E_\bsn\big[ \tfrac{1}{|\T|} \inf_{R\in\sod} \sum_{\t\in\T} \|V-R\|^2_{L^2(\t)}\big] &&\,=\, 0 \label{eq:lim2}\\
  \adjustlimits\lim_{\be\to\infty} \limsup_{N\to\infty} \sup_{\s\ge\s_0(N,m)} &
  E_\bsn\big[ \tfrac{1}{\la(\I)} \inf_{R\in\sod} \|V-R\|^2_{L^2(\I)}\big] &&\,=\, 0 \label{eq:lim3}
 \end{alignat}
\end{corol}
\begin{proof}
 Let $\de>0$. We define
 $$ f(\be,\de,m) := \be\tfrac{c_0}{8}\de + \cxxiv\log\be - \cxxv(\de,m) \,.$$
 Then $ \lim_{\be\to\infty}  f(\be,\de,m) = \infty$ for fixed $\de$ and $m$.
 Lemma~\ref{lem:EH-H} states that for all $\be\ge\be_0$ and $N\ge N_0(\de,m)$
 $$ \sup_{\s\ge\s_0(N,m)} E_\bsn\big[ \tfrac{1}{|\T|}\big(H_\sn(\cdot)-H_\sn(\sta)\big)\big]
    \lle 
    \de + \tfrac1\be e^{-N^d f(\be,\de,m)}
    \lle
    \de + \tfrac1\be e^{-f(\be,\de,m)} $$
 if $f(\be,\de,m)>0$ (which is fulfilled for large enough $\be$). Therefore
 $$ \adjustlimits\sup_{N\ge N_0(\de,m)} \sup_{\s\ge\s_0(N,m)} E_\bsn\big[ \tfrac{1}{|\T|}\big(H_\sn(\cdot)-H_\sn(\sta)\big)\big]   \lle  \de + \tfrac1\be e^{-f(\be,\de,m)} \,,$$
 which implies
 $$  \adjustlimits\limsup_{\be\to\infty} \limsup_{N\to\infty} \sup_{\s\ge\s_0(N,m)} 
  E_\bsn\big[ \tfrac{1}{|\T|}\big(H_\sn(\cdot)-H_\sn(\sta)\big)\big]
  \lle \lim_{\be\to\infty} \big(\de + \tfrac1\be e^{-f(\be,\de,m)}\big) \,=\, \de $$
 and therefore claim (\ref{eq:lim1}) with ``$\le 0$'' instead of ``$=0$'' and with ``$\limsup$'' instead of ``$\lim$'' since $\de>0$ was arbitrary.
 
 Lemma~\ref{lem:H-estimate} states that there exists $R(\om)$ such that
 $$ \inf_{R\in\sod} \|V(\om)-R\|^2_{L^2(\I)} \lle \|V(\om)-R(\om)\|^2_{L^2(\I)} \lle \tfrac1\cxix\big( H_\sn(\om)-H_\sn(\sta)\big) \,.$$
 Thus we can estimate using also Lemma~\ref{lem:estTPBoxNd}(\ref{asf})
 \begin{eqnarray*}
  0 &\le&  E_\bsn\big[ \tfrac{1}{|\T|} \inf_{R\in\sod} \sum_{\t\in\T} \|V-R\|^2_{L^2(\t)}\big] \\
  &\le&
  E_\bsn\big[ \tfrac{1}{\cxi\la(\I)} \inf_{R\in\sod} \|V-R\|^2_{L^2(\I)}\big]  \\
  &\le&
  E_\bsn\big[ \tfrac{1}{\cxi\cxv|\T|} \inf_{R\in\sod} \|V-R\|^2_{L^2(\I)}\big] \\
  &\le&
  \tfrac{1}{\cxi\cxv\cxix}\, E_\bsn\big[ \tfrac{1}{|\T|} \big( H_\sn(\cdot)-H_\sn(\sta)\big) \big] \,.
 \end{eqnarray*}
 Therefore, the already proven version of claim (\ref{eq:lim1}), namely the one with ``$\le 0$'' and ``$\limsup$'', implies the real claim (\ref{eq:lim1}) as well as claims (\ref{eq:lim2}) and (\ref{eq:lim3}).
\end{proof}
\begin{proof}[Proof of Theorem~\ref{thm:symmetrybreaking}]
 It is exactly statement \eqref{eq:lim2} of Corollary~\ref{cor:limits} above.
\end{proof}

\subsection{Some Concrete Models} \label{sec:concretemodels}

In this section, we want to give two concrete models to which we can apply the results of the previous sections. Thereto we have to choose all components stated in the beginning of Section~\ref{sec:modeldef}. First we consider a model on the triangular lattice which is an analogue to the model considered in \cite{hmr13}. Then we work with the $d$-dimensional cubic lattice. Other models can be constructed similarly.

\subsubsection{Two-dimensional Triangular Lattice}

As already stated, the following model is an analogue to \cite{hmr13}. Thus we work with their set-up and fix
\begin{enumerate}[\hspace{1.5em}(a)]
 \item 
 a real-valued potential function $\phi$ defined in an open interval containing $1$ such that $\phi$ is twice continuously differentiable with $\phi''>0$ and $\phi'(1)=0$,
 \item
 an $\al>0$ so small that $\phi$ is defined on $[1-\al,1+\al]$ and that \cite[Corollary~2.4]{hmr13} holds and
 \item
 an $\ell\in(1-\al/2,1+\al/2)$.
\end{enumerate}
This are almost literally the same assumptions as in \cite[page~3]{hmr13}. We only use the letter $\phi$ for the potential since $V$ has a different meaning here.

We identify $\C$ and $\R^2$ and work on the triangular lattice $A_2=\Z+\tau\Z$ with $\tau=e^{i\pi/3}$ and edges formed by nearest neighbours. In the following, we choose the components of our model.
\begin{enumerate}[\hspace{1.5em}1.]
 \item 
  Let us define the tessellation $\M$ of $\R^2$ first. All tiles will have the same type, i.e.\ $I=\{1\}$. Therefore we omit the superscript $i=1$ in the following. Let the standard tile $\st$ be the triangle with vertices $s_1:=0$, $s_2:=\ell1$ and $s_3:=\ell\tau$, i.e.\ 
  $$ \st := \big\{\la_1 \ell + \la_2 \ell\tau \mid \la_1.\la_2\ge0,\, \la_1+\la_2\le1 \big\} \,.$$
  Then the tessellation $\M$ is given by
  $$ \M := \big\{ \t=z+\xi\st \mid z\in\ell A_2, \xi\in\{1,\tau\}\big\}\,.$$
 \item We choose the parameter $\ep\in(0,\frac\al4)$ arbitrary.
 \item We choose the parameter $\rho\in(0,\frac{\ell}{3})$ arbitrary.
 \item We choose the parameter $\co>0$ arbitrary.
 \item
  The local Hamiltonian is induced by the potential $\phi$ and defined by
  \begin{alignat*}{2}
   \hloc :&&\:\, \Uept[] &\,\to\, \R \\
   &&\:\, \t=\hull\{x_1,x_2,x_3\} &\,\mapsto\, \2\big(\phi(|x_1-x_2|)+\phi(|x_2-x_3|)+\phi(|x_3-x_1|)\big) \,
  \end{alignat*}
  where $x_1$, $x_2$, $x_3$ are the corners of $\t$. Since $|x_1-x_2|\le|x_1-s_1|+|s_1-s_2|+|s_2-x_2|\le\ell+2\ep<1+\al$ and similarly $|x_1-x_2|\ge1-\al$ and for the other vertex-pairs we conclude that $\hloc$ is well-defined. Moreover, it inherits continuity from $\phi$. Lemma~\ref{lem:conHlocfortriangle} below shows that inequality \eqref{eq:conHloc} is fulfilled.
 \item
  Finally we define the quantity $S$ measuring the surface of the crystal by
  $$ S := |\del\P| $$
  such that condition \eqref{eq:conS} is obviously fulfilled (with $\ciii=1$).
\end{enumerate}
The upcoming lemma shows that the local Hamiltonian indeed fulfils inequality \eqref{eq:conHloc}.
\begin{lem} \label{lem:conHlocfortriangle}
 There are constants $\ci>0$ and $\cii\in\R$ (depending on $\phi$) such that inequality \eqref{eq:conHloc} holds for all $\t\in\Uept[]$, i.e.\
 $$ H_\text{\emph{loc}}(\t)-H_\text{\emph{loc}}(\st)\,\ge\, \ci \|\dist(\nabla v_{\scriptscriptstyle\t},\sot)\|^2_{L^2(\t)} + \cii\big(\la(\t)-\la(\st)\big)\,,$$
 where $v_{\scriptscriptstyle\t}$ is the affine linear map mapping $\t$ to $\st$.
\end{lem}
\begin{proof}
 This is a more or less direct consequence of Corollary~2.4 in \cite{hmr13}. Let $x_1$, $x_2$ and $x_3$ be the corners of $\t\in\Uept$. By the definition of $\hloc$, we have
 \begin{equation} \label{eq:conHlocfortriangle1}
  \hloc(\t)-\hloc(\st) = \2\big(\phi(|x_1-x_2|)+\phi(|x_2-x_3|)+\phi(|x_3-x_1|)-3\phi(\ell)\big)\,.
 \end{equation}
 Moreover, \cite[Corollary~2.4]{hmr13} states (in our notation)
 \begin{equation} \label{eq:conHlocfortriangle2}
  \phi(|x_1-x_2|)+\phi(|x_2-x_3|)+\phi(|x_3-x_1|)-3\phi(\ell)-p(\ell)\big(\la(\t)-\la(\st)\big) \asymp_\phi \dist(\ell^{-1}\nabla\om,\sot)^2
 \end{equation}
 where $p(\ell)=2\sqrt{3}\phi'(\ell)/\ell$ and $\om$ is the affine linear map mapping $0\mapsto x_1$, $1\mapsto x_2$ and $\tau\mapsto x_3$. Since $v_{\scriptscriptstyle\t}$ is the affine linear map mapping $x_1\mapsto 0$, $x_2\mapsto \ell$ and $x_3\mapsto \ell\tau$, we conclude 
 \begin{equation} \label{eq:conHlocfortriangle3}
  v_{\scriptscriptstyle\t}\circ\om=\ell\operatorname{Id} \qquad\text{and therefore}\qquad  \ell^{-1}\nabla\om = (\nabla v_{\scriptscriptstyle\t})^{-1} \,.
 \end{equation}
 Now we use the following fact: For all $A\in\R^{2\times2}$ which are close to $\sot$ one has
 $$  \dist(A^{-1},\sot)^2 \asymp \dist(A,\sot)^2 \,.$$
 Applying this fact to $A=\nabla v_{\scriptscriptstyle\t}$, which is close to $\sot$ as $|x_j-s_j|\le\ep$ ($j=1,2,3$), yields
 \begin{equation} \label{eq:conHlocfortriangle4}
  \dist(\ell^{-1}\nabla\om,\sot)^2 \asymp \dist(\nabla v_{\scriptscriptstyle\t},\sot)^2 \,.
 \end{equation}
 since $(\nabla v_{\scriptscriptstyle\t})^{-1}=\ell^{-1}\nabla\om$ by \eqref{eq:conHlocfortriangle3}. Combining equations \eqref{eq:conHlocfortriangle1}, \eqref{eq:conHlocfortriangle2}, \eqref{eq:conHlocfortriangle4} and $\la(\t)\asymp1$ yields the lemma.
\end{proof}
We recall the definition of the Hamiltonian and the probability measure in the end of Section~\ref{sec:modeldef}. Thereto let $\be>0$, $\s>0$, $m\ge m_0$ and $N\in\N$. We define the Hamiltonian
$$ H_\sn\,:=\, \sum_{\t\in\T} \hloc(\t) + \s S - m |\P| $$
and the probability measure $P_\bsn$ via
$$ \frac{dP_\bsn}{d\mu} \,:=\, \frac{1}{Z_\bsn} e^{-\be H_\sn} \qquad\text{with}\quad Z_\bsn \,:=\, \int_\Om e^{-\be H_\sn} \,d\mu \,.$$
One may be bothered by the fact that edges inside the crystal $\T$ appear twice in the Hamiltonian whereas boundary edges appear only once. But this disturbance can be fixed using the following alternative tilde-versions. Let us define the Hamiltonian 
$$ \ti{H}_\sn\,:=\, \sum_{\substack{x,y\in\P\\ x\sim y \text{ in}\,\T}} \phi\big(|x-y|\big) + \s |\del\P| - m |\P| \,.$$
where $x\sim y$ in $\T$ iff there exists $\t\in\T$ with $x,y\in\t$ and $x\ne y$. Then the probability measure $\ti{P}_\bsn$ is defined via
$$ \frac{d\ti{P}_\bsn}{d\mu} \,:=\, \frac{1}{\ti{Z}_\bsn} e^{-\be \ti{H}_\sn} \qquad\text{with}\quad \ti{Z}_\bsn \,:=\, \int_\Om e^{-\be \ti{H}_\sn} \,d\mu \,.$$
We denote the expectation with respect to $P_\bsn$ with $E_\bsn$ and the expectation with respect to $\ti{P}_\bsn$ with $\ti{E}_\bsn$.

Then we have the following corollary to Theorem~\ref{thm:symmetrybreaking}.
\begin{corol} \label{cor:symmetrybreakingtriangle}
 There exist $m_0\in\R$ and $\s_0(N,m)\asymp N^2 + m$ such that the rotational symmetry of the crystal is broken in the following sense: 
 $$\forall\, m\ge m_0:\quad\adjustlimits\lim_{\be\to\infty} \limsup_{N\to\infty} \sup_{\s\ge\s_0(N,m)} 
  E_\bsn\bigg[ \inf_{R\in\sot} \frac{1}{|\T|}\sum_{\t\in\T} \|V-R\|^2_{L^2(\t)}\bigg] \;=\; 0 $$
 as well as
 $$\forall\, m\ge m_0:\quad\adjustlimits\lim_{\be\to\infty} \limsup_{N\to\infty} \sup_{\s\ge\s_0(N,m)} 
  \ti{E}_\bsn\bigg[ \inf_{R\in\sot} \frac{1}{|\T|}\sum_{\t\in\T} \|V-R\|^2_{L^2(\t)}\bigg] \;=\; 0 $$
 holds.
\end{corol}
\begin{proof}
 For $E_\bsn$, this is exactly the statement of Theorem~\ref{thm:symmetrybreaking}. For $\ti{E}_\bsn$, we observe that
 $$ \big|H_\sn-\ti{H}_\sn\big| \lle \Big|\!\sum_{\substack{x,y\in\del\P\\ x\sim y \text{ in}\,\T}} \!\!\2\phi\big(|x-y|\big)\Big| \lle \clv |\del\P|=\clv S$$
 with $\clv:=3\sup_{t\in[1-\al,1+\al]}|\phi(t)|$. Therefore
 $$ \ti{H}_\sn \gge H_\sn-\clv S \,=\, \sum_{\t\in\T}\hloc(\t)+(\s-\clv)S-m|\P| \,=\, H_{\s-\clv,m,N} $$
 and analogously
 $$ \ti{H}_\sn \lle H_{\s+\clv,m,N} \,.$$
 Thus $\ti{Z}_\bsn \ge Z_{\be,\s+\clv,m,N}$ and
 $$ \ti{E}_\bsn\big[ \inf_{R\in\sot} \tfrac{1}{|\T|}\sum_{\t\in\T} \|V-R\|^2_{L^2(\t)}\big]
   \lle \int_\Om \frac{e^{-\be H_{\s-\clv,m,N}}}{Z_{\be,\s+\clv,m,N}} \inf_{R\in\sot} \tfrac{1}{|\T|}\sum_{\t\in\T} \|V-R\|^2_{L^2(\t)} \,d\mu $$
 Moreover, we observe that the lower bound of the partition sum in Lemma~\ref{lem:partitionsum} does not depend on $\s$ since $H_\sn(\sta)$ is independent of $\s$. Thus we can apply the proof of Theorem~\ref{thm:symmetrybreaking} for $P_{\be,\s-\clv,m,N}$ if $\s-\clv\ge\s_0(N,m)$ to conclude that the appropriate limit of the right hand side is $0$. Therefore the corollary for $\ti{E}_\sn$ follows if we enlarge $\s_0(N,m)$ by $\clv$.
\end{proof}

\subsubsection{Cubic Lattice in $d$ Dimensions}
Finally we give an example on the cubic lattice in dimension $d\ge2$. First we note that a cube is not stabilized by fixing all its edge lengths: it can be arbitrarily flat. Thus there is no chance to be close to $\sod$ if only the edge lengths are specified. Therefore we specify the lengths of the diagonals, too. Though not required, we use all diagonals in order to simplify the presentation. The following model is quite similar to the model on the triangular lattice; thus we do not present all technical details.

We define $D:=\{A\subset\{1,\ldots,2^d\}:|A|=2\}$. This set is used to index a pair or ``double'' of vertices of a cube, or the corresponding edge or diagonal. We shortly write $kj\in D$ for $\{k,j\}\in D$. Similarly as for the model on the triangular lattice, we fix
\begin{enumerate}[\hspace{1.5em}(a)]
 \item 
 a tuple of real-valued potential functions $\phi_{kj}$, $kj\in D$, defined in an open interval containing $1$ such that each $\phi_{kj}$ is twice continuously differentiable with $\phi_{kj}''>0$ and $\phi_{kj}'(1)=0$,
 \item
 an $\al>0$ so small that each $\phi_{kj}$ is defined on $[1-\al,1+\al]$ and that Lemma~\ref{lem:conHloccubic} below holds and
 \item
 an $\ell\in(1-\al/2,1+\al/2)$.
\end{enumerate}
Using this input, we define the model according to the set-up in Section~\ref{sec:modeldef}. First we choose the parameters $\ep\in(0,\frac\al4)$, $\rho\in(0,\frac{\ell}{3})$ and $\co>0$ arbitrary. The tessellation $\M$ will be induced by the lattice $\ell\Z^d$. Again there is only one tile type such we can omit the superscript $i$. The standard tile $\st$ is the cube $\st=\{(z_1,\dots,z_d)\in\R^d\mid 0\le z_1,\dots,z_d \le \ell\}$; its corners are denoted by $s_1,\ldots,\cramped{s_{2^d}}$. Then $\M:=\{z+\st\mid z\in \ell\Z^d\}$. If a perturbed cube $\t\in\Uept[]$ has corners $x_1,\ldots,\cramped{x_{2^d}}$, we define its local Hamiltonian using the given potential functions as follows:
$$ \hloc(\t) \,:=\, \sum_{kj\in D} \phi_{kj}\left(\frac{|x_k-x_j|}{\cramped{\ell^{-1}}|s_k-s_j|}\right) \,.$$
Thus we allow different potentials for different edges or diagonals. Similarly to the example on the triangular lattice we conclude that $\hloc$ is well-defined and continuous; Lemma~\ref{lem:conHloccubic} below shows that inequality \eqref{eq:conHloc} is fulfilled. Again we define the quantity $S$ measuring the surface of the crystal by
  $ S := |\del\P|$
such that condition \eqref{eq:conS} is obviously fulfilled. We still need
\begin{lem} \label{lem:conHloccubic}
 For sufficiently small $\al>0$, there are constants $\ci>0$ and $\cii\in\R$ such that inequality \eqref{eq:conHloc} holds for all $\t\in\Uept[]$, i.e.\
 $$ H_\text{\emph{loc}}(\t)-H_\text{\emph{loc}}(\st)\,\ge\, \ci \|\dist(\nabla v_{\scriptscriptstyle\t},\sod)\|^2_{L^2(\t)} + \cii\big(\la(\t)-\la(\st)\big)\,,$$
 where $v_{\scriptscriptstyle\t}$ is the affine linear map mapping $\t$ to $\st$.
\end{lem}
\begin{proof}
 The proof is quite similar to the proofs of Lemma~2.2, Lemma~2.3 and Corollary~2.4 in \cite{hmr13}. In fact, it generalises their arguments to higher dimensions. Therefore, we present not all technical details.

 Let a tile $\t\in\Uept[]$ with corners $x_1,\ldots,x_{\cramped{2^d}}$ be given. We abbreviate
 $$ \xi_{kj}:=\frac{|x_k-x_j|}{\cramped{\ell^{-1}}|s_k-s_j|} $$
 for $kj\in D$. There exists a twice continuously differentiable function
 $$f:\R_+^{|D|}\to\R \qquad\text{ with }\qquad \la(\t)=f\big(\xi_{kj}:kj\in D\big) $$
 for $\t\in\Uept[]$. Using a Taylor expansion around $(\ell,\ldots,\ell)$, we conclude
 $$ \la(\t)-\la(\st) = \sum_{kj\in D} \del_{kj}f(\ell,\dots,\ell)\,(\xi_{kj}-\ell) + O\big(\textstyle\sum\displaystyle(\xi_{kl}-\ell)^2\big) \,.$$
 Note that $b:=\inf_{kl}\del_{kj}f(\ell,\dots,\ell)>0$ since increasing an edge length increases the volume. It follows that
 $$ b\sum_{kj\in D} |\xi_{kj}-\ell| \lle \sum_{kj\in D} \del_{kj}f(\ell,\dots,\ell)\,|\xi_{kj}-\ell| \lle
   |\la(\t)-\la(\st)| + O\big(\textstyle\sum\displaystyle(\xi_{kl}-\ell)^2\big)\,.$$
 Now we use $\sup\!|\phi_{kj}'(l)|\le\al\sup\!|\phi_{kj}''(l)|$, where the suprema are taken over all $kj\in D$ and $l\in[1-\al/2,1+\al/2]$, to conclude
 \begin{eqnarray*}
  \sum_{kj\in D} \phi'(\ell)(\xi_{kj}-\ell)
  &\ge& -\sup\!|\phi_{kj}'(l)|\sum_{kj\in D} |\xi_{kj}-\ell| \\
  &\ge& 
  -\tfrac{\al}{b}\sup\!|\phi_{kj}''(l)| \Big(\la(\t)-\la(\st) + O\big(\textstyle\sum\displaystyle(\xi_{kl}-\ell)^2\big)\Big)\,.
 \end{eqnarray*}
 Note that we can ignore the absolute value of $|\la(\t)-\la(\st)|$ since $\la(\t)\ge\la(\st)$ holds for all $\t\in\Uept[]$. Applying Taylor's Theorem to $\phi_{kj}$, $kj\in D$, yields with the just obtained estimate
 \begin{eqnarray}
  \hloc(\t)-\hloc(\st)
  &=&
  \sum_{kj\in D} \Big(\phi_{kj}(\xi_{kj})-\phi_{kj}(\ell)\Big) \notag\\
  &=&
  \sum_{kj\in D} \Big(\phi_{kj}'(\ell)\,(\xi_{kj}-\ell) + \2\phi_{kj}''(\ell)\,(\xi_{kj}-\ell)^2 + o\big((\xi_{kj}-\ell)^2\big)\Big) \notag\\
  &\ge&
  -\tfrac{\al}{b}\sup\!|\phi_{kj}''(l)| \big(\la(\t)-\la(\st)\big)+ \inf\!\big[\2\phi_{kj}''(l)\big]\sum\nolimits_{kj\in D}(\xi_{kj}-\ell)^2 \notag\\
  && +\, o\big(\textstyle\sum\displaystyle(\xi_{kj}-\ell)^2\big) - \tfrac{\al}{b}\sup\!|\phi_{kj}''(l)|\,O\big(\textstyle\sum\displaystyle(\xi_{kl}-\ell)^2\big) \notag\\
  &\ge&
  \cii \big(\la(\t)-\la(\st)\big) + \clvi \sum_{kj\in D}(\xi_{kj}-\ell)^2 \label{eq:last}
 \end{eqnarray}
 with $\cii=-\al\sup\!|\phi_{kj}''(l)|/b\in\R$ and some constant $\clvi>0$ for small enough $\al>0$ since $\inf\2\phi_{kj}''(l)>0$, where the infimum is taken over all $kj\in D$ and $l\in[1-\al/2,1+\al/2]$.
 
 It remains to bound
 $\sum_{kj\in D}(\xi_{kj}-\ell)^2$ in terms of $\|\dist(\nabla v_{\scriptscriptstyle\t},\sot)\|^2_{L^2(\t)}$.
 Thereto we consider any simplex $\triangle\subset\t$ such that $v_{\scriptscriptstyle\t}$ is affine linear on $\triangle$. Let $\ti{D}=\ti{D}_\triangle\subset D$ denote the corresponding set of vertex pairs of the simplex. Let $kj\in\ti{D}$. Setting $M:=\cramped{(\nabla v_{\scriptscriptstyle\t})^{-1}}$, which is constant on $\triangle$, yields $x_k-x_j=M(s_k-s_j)$ as $v_{\scriptscriptstyle\t}$ maps $x_k$ to $s_k$ and $x_j$ to $s_j$. Using also $\ell^{-2}|s_k-s_j|^2\asymp1$ we conclude
 \begin{eqnarray*}
  \xi_{kj}-\ell
  &\asymp&
  \ell^{-2}|s_k-s_j|^2 (\xi_{kj}^2-\ell^2)
  \,=\,
  |x_k-x_j|^2-|s_k-s_j|^2 \\
  &=&
  |M(s_k-s_j)|^2-|s_k-s_j|^2
  \,=\,
  \langle (s_k-s_j), (M^*M-\operatorname{Id})(s_k-s_j) \rangle
 \end{eqnarray*}
 We define a norm $\|Q\|_s$ of a symmetric $d\times d$-matrix $Q$ by
 $$ \|Q\|_s := \sqrt{\sum_{kj\in\ti{D}}\langle (s_k-s_j), Q(s_k-s_j) \rangle^2} \,.$$
 This is obviously a semi-norm; since $(s_k-s_j)$, $kj\in\ti{D}$, are the edges of a simplex, it even is a norm. As in the proof of \cite[Lemma~2.3]{hmr13} we conclude $\|M^*M-\operatorname{Id}\|_s\asymp\dist(M,\sod)$. Thus we have shown that
 $$ \sum_{kj\in\ti{D}} (\xi_{kj}-\ell)^2 \,\asymp\,\|M^*M-\operatorname{Id}\|_s^2\,\asymp\,\dist(M,\sod)^2\,\asymp\,\dist(\nabla v_{\scriptscriptstyle\t},\sod)^2 $$
 since $\nabla v_{\scriptscriptstyle\t}=M^{-1}$ is close to $\sod$ (for small $\al$) because $\t$ is an $\ep$-perturbation of $\st$. Using the facts that the Lebesgue measure of any simplex of $\t$ is of order $1$ and that each diagonal belongs only to a finite number of simplexes, we conclude
 \begin{eqnarray*}
  \sum_{kj\in D} (\xi_{kj}-\ell)^2
  &\gtrsim& \adjustlimits\sum_{\triangle} \sum_{kj\in\ti{D}_\triangle}\! (\xi_{kj}-\ell)^2 \,\asymp\, \sum_{\triangle} \|\dist(\nabla v_{\scriptscriptstyle\t},\sod)\|^2_{L^2(\triangle)} \\
  &=& \|\dist(\nabla v_{\scriptscriptstyle\t},\sod)\|^2_{L^2(\t)} \,.
 \end{eqnarray*}
 Inserting this inequality into \eqref{eq:last} completes the proof.
\end{proof}

It follows that all assumptions in Section~\ref{sec:modeldef} are fulfilled. Therefore the very last corollary needs no further proof.
\begin{corol}
 The rotational symmetry of the crystal model on the cubic lattice introduced above is broken in the sense of Theorem~\ref{thm:symmetrybreaking}. \hfill\qed 
\end{corol}

\subsection*{Acknowledgement}
The author is grateful to Franz Merkl for stimulating discussions and helpful remarks. This research was supported by a scholarship of the Cusanuswerk, one of the German national academic foundations.

\end{document}